\documentclass[12pt,leqno]{amsart}
\usepackage{amssymb,amsmath,amsthm,bm}
\usepackage{graphicx,xparse}
\usepackage{color}
\usepackage[textsize=tiny]{todonotes}
\usepackage[normalem]{ulem}
\usepackage[bookmarksopen,bookmarksdepth=3,colorlinks,citecolor=red,pagebackref,hypertexnames=true]{hyperref}
\usepackage[msc-links,nobysame,non-sorted-cites, initials]{amsrefs}
\usepackage[inline]{enumitem}
\usepackage{mathtools}
\mathtoolsset{showonlyrefs}
\usetikzlibrary{quotes,arrows.meta}

\definecolor{darkblue}{RGB}{0,0,160}

\usepackage[lining]{libertine}   
\usepackage{cabin}
\usepackage[libertine]{newtxmath}

\usepackage[T1]{fontenc}

\def\e{{\rm e}}
\def\eps{\varepsilon}

\def\d{{\rm d}}
\def\dist{{\rm dist}}

\def\s {{\mathsf s}}

\def\R {\mathbb{R}}
\def\N {\mathbb{N}}
\DeclareMathOperator*\outsup{outsup}

\def\Mod {{\mathrm{Mod}}}
\def\Dil {{\mathrm{Dil}}}

\def\D {{\mathcal D}}

\def\J {{\mathbb J}}
\def\I {{\mathcal I}}

\def\F {{\mathcal F}}

\def\S {\mathcal{S}}

\def\size {{\mathsf{size}}}

\def\supp {{\mathrm{supp}\,}}
\def\dense {{\mathsf{dense}}}
\def\T {{\mathcal T}}

\def\M {{\mathrm M}}

\def\Z {{\mathbb Z}}
\def\1 {{\mbox{\boldmath 1}}}

\def \l {\langle}

\def \r {\rangle}

\def \scl{\mathsf{scl}}

\def \and{\quad\text{and}\quad}

\newcommand{\cic}{\bm}

\newcommand{\ds}{\displaystyle}

\newcommand{\ch}{\mathsf{ch}}

\def\Xint#1{\mathchoice
	{\XXint\displaystyle\textstyle{#1}}%
	{\XXint\textstyle\scriptstyle{#1}}%
	{\XXint\scriptstyle\scriptscriptstyle{#1}}%
	{\XXint\scriptscriptstyle\scriptscriptstyle{#1}}%
	\!\int}
\def\XXint#1#2#3{{\setbox0=\hbox{$#1{#2#3}{\int}$}
		\vcenter{\hbox{$#2#3$}}\kern-.5\wd0}}

\def\avgint{\Xint-}

\makeatletter
\DeclareFontFamily{OMX}{MnSymbolE}{}
\DeclareSymbolFont{MnLargeSymbols}{OMX}{MnSymbolE}{m}{n}
\SetSymbolFont{MnLargeSymbols}{bold}{OMX}{MnSymbolE}{b}{n}
\DeclareFontShape{OMX}{MnSymbolE}{m}{n}{
    <-6>  MnSymbolE5
   <6-7>  MnSymbolE6
   <7-8>  MnSymbolE7
   <8-9>  MnSymbolE8
   <9-10> MnSymbolE9
  <10-12> MnSymbolE10
  <12->   MnSymbolE12
}{}
\DeclareFontShape{OMX}{MnSymbolE}{b}{n}{
    <-6>  MnSymbolE-Bold5
   <6-7>  MnSymbolE-Bold6
   <7-8>  MnSymbolE-Bold7
   <8-9>  MnSymbolE-Bold8
   <9-10> MnSymbolE-Bold9
  <10-12> MnSymbolE-Bold10
  <12->   MnSymbolE-Bold12
}{}

\let\llangle\@undefined
\let\rrangle\@undefined
\DeclareMathDelimiter{\llangle}{\mathopen}%
                     {MnLargeSymbols}{'164}{MnLargeSymbols}{'164}
\DeclareMathDelimiter{\rrangle}{\mathclose}%
                     {MnLargeSymbols}{'171}{MnLargeSymbols}{'171}
\makeatother

\def \no#1#2#3 {{\bf #1} (#3), #2.}
\def \eds#1#2#3 {#1, #2, #3.}

\newcounter{counter}

\numberwithin{equation}{section}
\numberwithin{counter2}{section}
\newtheorem{proposition}[subsection]{Proposition}
\newtheorem{theorem}[counter]{Theorem}

\newtheorem{corollary}{Corollary}
\newtheorem{lemma}[subsection]{Lemma}

\theoremstyle{definition}

\newtheorem*{remark*}{Remark}
\newtheorem*{warn*}{A word of warning}

\newtheorem{remark}[subsection]{Remark} 
\theoremstyle{plain}
\numberwithin{corollary}{counter}

\newcommand{\asqref}[1]{\textup{\tagform@{\ref*{#1}}}}
 
\numberwithin{figure}{section}

\begin{document}

	\title[Weak-type Carleson theorem]{The weak-type  Carleson theorem via wave packet estimates} 
	
	\author[F.\ Di Plinio]{Francesco Di Plinio} 
\address[F.\ Di Plinio]{Dipartimento di Matematica e Applicazioni, Universit\`a di Napoli \\ \newline \indent Via Cintia, Monte S.\ Angelo 80126 Napoli, Italy}
\email{\href{mailto:francesco.diplinio@unina.it}{\textnormal{francesco.diplinio@unina.it}}}
	
	\author[A. Fragkos]{Anastasios Fragkos}\address[A. Fragkos]{Department of Mathematics, Washington University in Saint Louis\\ \newline \indent 1 Brookings Drive, Saint Louis, Mo 63130, USA}
\email{\href{mailto:anastasiosfragkos@wustl.edu}{\textnormal{anastasiosfragkos@wustl.edu}}}

	\thanks{F. Di Plinio and A.\ Fragkos  have been partially supported by the National Science Foundation under the grants   NSF-DMS-2000510, NSF-DMS-2054863}

	
	\maketitle
	\begin{abstract} We prove that  the weak-$L^{p}$ norms, and in fact the sparse $(p,1)$-norms, of the  Carleson maximal partial Fourier sum operator are
	 $\lesssim (p-1)^{-1}$ as $p\to 1^+$. This is an improvement on the  Carleson-Hunt theorem, where the same upper bound on the growth order is obtained for the restricted weak-$L^p$ type norm, and which was the strongest quantitative bound prior to our result. Furthermore, our sparse $(p,1)$-norms bound  imply new and stronger results   at the endpoint $p=1$. In particular, we obtain that the Fourier series of functions  from the weighted Arias de Reyna space $ \mathrm{QA}_{\infty}(w) $, which contains the  weighted Antonov space  $L\log L\log\log\log L(\mathbb T; w)$,   converge almost everywhere whenever $w\in A_1$. This    is an  extension of the results of    Antonov and Arias De Reyna,  where $w$  must be  Lebesgue measure.
	  	 	
	The  backbone of our treatment is a new, sharply quantified near-$L^1$ Carleson embedding theorem for the modulation-invariant wave packet transform. The proof of the Carleson embedding relies on a newly developed smooth multi-frequency decomposition which, near the endpoint $p=1$, outperforms the abstract Hilbert space  approach of past works, including the seminal one by Nazarov, Oberlin and Thiele. As a further example of application, we obtain a quantified version of the  family of sparse bounds for the bilinear Hilbert transforms due to Culiuc, Ou and the first author.	  	\end{abstract}
	
	\section{Introduction and main results}
The fundamental question of whether the Fourier series   of a square integrable function
 on the torus $\mathbb T \coloneqq \R / (2\pi\mathbb Z) $ converges Lebesgue a.e.\ $ x\in\mathbb T$ was answered in the affirmative by L.\ Carleson in 1966 \cite{C}, by means of a weak-$L^2$ inequality for the maximal operator
\begin{equation}
\label{carleson0}
\mathcal C f(x) =\sup_{N\in \mathbb Z} \Bigg| \sum_{|\xi|\leq N} \widehat f(\xi) \exp(ix\xi) \Bigg|, \qquad x\in \mathbb T.
\end{equation}
The rich and surprising argument of \cite{C} estimates  $\mathcal C $  pointwise as a maximal modulated Hilbert transform,  outside suitably constructed exceptional sets whose  mass is controlled   by almost-orthogonality. The  distributional estimate  implicit in \cite{C}  was later exploited by Hunt \cite{H} to deduce the family of restricted weak-type $L^p$ bounds
\begin{equation}
	\label{e:ch}	\left\|\mathcal{C}f\right\|_{L^{p,\infty}(\mathbb T)} \leq \frac{Cp^2}{p-1}|F|^{\frac1p}, \qquad F\subset \mathbb T,\; |f|\leq \cic{1}_F, \qquad 1<p<\infty.	\end{equation} 
The estimate \eqref{e:ch} and   interpolation yield that $\mathcal{C}$ is a bounded operator on each $L^p(\mathbb T),$ for $1<p<\infty$. Consequently,  pointwise a.e.\ convergence of the Fourier series holds for $f\in L^p(\mathbb T)$ in the same range. Since \cite{C,H}, several substantially different proofs of Carleson's theorem have appeared: in particular, the celebrated ones by Fefferman \cite{F} and Lacey-Thiele \cite{LT},  one implicit in the return times theorem of Demeter, Lacey, Tao and Thiele \cite{DLTT}, and more recently an improvement of Fefferman's proof \cite{F} due to Lie \cite{Lie13}.  

The primary focus of our work is the behavior of the Carleson operator as $p\to 1^+$. Besides its intrinsic interest,  this question is deeply connected to the pointwise a.e.\ behavior  of  Fourier series in function spaces between $L^1(\mathbb T)$ and $L^p(\mathbb T)$. To exemplify the connection,  Antonov \cite{Antonov96} coupled  the precise information on the growth rate of the restricted weak norm from \eqref{e:ch}   with an approximation argument to deduce  a mixed type estimate, which is the case $w=1$ of \eqref{e:mteantonov} below. This may be leveraged to extend the pointwise convergence result to functions in the Orlicz space $ L\mathrm{\log}_1 L\mathrm{log}_3 L(\mathbb T)$. Antonov's result has been, to date, the strongest known within the Orlicz-Lorentz scale; see    Remark \ref{r:prev} for a more precise statement and \cite{DPCRM} for a thorough discussion of the interplay of weak type and endpoint bounds.

 To this purpose,   all  proofs of Carleson's theorem to date yield  estimates  at best quantitatively equivalent to \eqref{e:ch}. Citing from the preface to Arias de Reyna's  2002 lecture notes, ``To this day [Carleson's] is the proof that gives the finest results about the maximal operator of Fourier series.'' \cite[p. V]{ADRbook}. This appears to be the case even after accounting for the substantial progress in the understanding of modulation invariant singular integrals,   in particular of Carleson's operator, that has occurred in the past twenty years; see e.g.\ \cite{DoObPalss2017,DoLac2012,HytLac,Lie2017,LieStrong2019,Lie2020,NOT2010,OSTTW,Z} in addition to the above mentioned \cite{DLTT,Lie13}. 

\subsection{Main results within the state of the art}The motivating result of the present article  goes beyond  the Carleson-Hunt bound \eqref{e:ch}, upgrading the estimate to the weak $L^p$-type.
\begin{theorem}  \label{cor:carl2}
The maximal operator \eqref{carleson0} obeys the family of estimates
\[
\left\|\mathcal{C}f\right\|_{L^{p,\infty}(\mathbb T)} \leq \frac{C}{p-1}\left\|f\right\|_{L^{p}(\mathbb T)}, \qquad 1<p\leq 2.
\]
The same bounds hold for the   maximal multiplier \eqref{e:mainspper} and for   the real line  analogue \eqref{carleson}.\end{theorem}
 In fact, we obtain Theorem \ref{cor:carl2} as an immediate corollary of a stronger quantitative estimate for the \emph{sparse norms} of the operator $\mathcal C$, described as follows. For $n\geq 2$,  $\vec p=(p_1,\ldots,p_n) \in (0,\infty)^n$,  the  $n$-linear $\vec p$-maximal function  of a tuple $\{f_j\in L^{p_j}_{\mathrm{loc}}(\R^d):1\leq j\leq n\}$ is defined as
\[
\mathrm{M}_{\vec p}(f_1,\ldots, f_n) \coloneqq \sup_{Q } \cic{1}_Q \prod_{j=1}^n \langle f_j\rangle_{p_j,Q}
\]
the supremum being taken over all cubes $Q$ of $\R^d$. See the final paragraph of this introduction for a summary of standard notations. An   $n$-sublinear form  $\Lambda$   acting e.g. on $n$-tuples of functions $f_j\in L^\infty_0(\R^d)$ is $\vec p$-\emph{sparse bounded} if there exists a constant $C>0$ such that
\[
\left| \Lambda (f_1,\ldots, f_n) \right| \leq C \left \| \mathrm{M}_{\vec p}(f_1,\ldots, f_n) \right \|_1
\]
uniformly over all such tuples, and the $\vec p$-\emph{sparse bound} $\|\Lambda\|_{\vec p}$ is the infimum of the set of all such constants. If $T$ is an $(n-1)$-sublinear operator, the quantity $\|T\|_{\vec p}$ indicates the sparse bound $\|\Lambda\|_{\vec p}$ of the $n$-sublinear form
\[
\Lambda (f_1,\ldots, f_n)=\langle T(f_1,\ldots, f_{n-1}), f_n \rangle.
\]
Note that $T$ is a specific formal adjoint of $\Lambda$, and the index $n$ plays a distinguished role.
The equivalence of this formulation with more standard notions of sparse bounds \cite{LN} is thoroughly discussed in \cite{CDOBP,NieZ2021} and references therein.
\begin{theorem} \label{t:A}
 Let $m\in L^\infty(\R) \cap \mathcal C^\infty(\R\setminus\{ 0\}) $ be a smooth H\"ormander-Mihlin multiplier, see \eqref{e:HM0}.
The associated maximally modulated multiplier  
  \begin{equation}
	\label{carleson}	\mathcal{C}f(x)\coloneqq 	\sup_{N \in \R} \left| \int_\R m(\xi-N)\widehat{f}(\xi)  \e^{i x \xi} \, \d \xi\right| \qquad x \in \R,
	\end{equation}
satisfies the family of sparse bounds
	  \begin{equation}
	\label{e:mainsp1}	\left\|\mathcal{C}\right\|_{(p,1)} \leq \frac{C}{p-1}, \qquad 1<p\leq 2	\end{equation}
with a uniform constant $C$. The same estimates hold for the periodic version of \eqref{carleson}
   \begin{equation}
	\label{e:mainspper}	\mathcal{C}f(x)\coloneqq 	\sup_{N \in \mathbb Z} \left| \sum_{\xi \in  \mathbb Z} m(\xi-N)\widehat{f}(\xi)  \e^{i x \xi} \right| \qquad x \in \mathbb T
	\end{equation}
	under the additional transference assumption that ${\displaystyle\lim_{\eps\to 0^+} } \;\avgint_{|t|<\eps} m$ exists.
\end{theorem}
Note that the Carleson maximal operator, on the real line and on the torus respectively, correspond up to symmetries and linear combination with the identity operator to the choice $m= \cic{1}_{(0,\infty)}$ in \eqref{carleson}, \eqref{e:mainspper}. 

The $\vec p$-sparse bounds of $T$ subsume a full range of quantitative weighted norm inequalities of weak and strong type.  We send to the references \cite{LMO,NieZ2021} for a complete list of consequences and for the related extrapolation theory,  and content ourselves with recalling those implications most crucial for our exposition, in the form of corollaries to this main result. Then, the estimates of Theorem \ref{cor:carl2} are derived from the sparse bound of Theorem \ref{t:A} as in e.g.\ \cite[Theorem E]{CoCuDPOu}. Two more corollaries     are of weighted nature. 
The first  is a weighted version of Theorem \ref{cor:carl2}; to wit, a weak type $L^p(w)$ bound for $A_1$ weights with controlled constant. In light of Corollary \ref{cor:a2} below, this is interesting  when $p \to 1^+$ and entails certain weighted Carleson estimates at the endpoint $p=1$ as well. These estimates imply  pointwise convergence of Fourier series for a class which is strictly larger than those of \cite{Antonov96,ADR}.
\begin{corollary} \label{cor:weighted} For  weights $w\in A_1$ and  $1\leq p\leq 2$, define \[\mathsf{K}(w,p)\coloneqq [w]_{A_1}^{\frac1p} [w]_{A_\infty}^{1-\frac	1p} [\mathrm{log}_1[w]_{A_\infty}]^{\frac2p}.\] For  both \eqref{carleson}, \eqref{e:mainspper}, there holds
\begin{equation}
\label{e:weighted1}
\left\| \mathcal{C} f \right\|_{L^{p,\infty}(w)} \leq \frac{C\mathsf{K}(w,p)}{p-1}  \left\|  f \right\|_{L^{p}(w)}, \qquad 1<p\leq 2.
\end{equation}
As a  further corollary of \eqref{e:weighted1},   \eqref{e:mainspper} satisfies the following endpoint estimates:
\begin{align}
&  \label{e:mteantonov} \left\| \mathcal{C} f \right\|_{L^{1,\infty}(\mathbb T; w)} 
\leq \frac{Cq}{q-1} \mathsf{K}(w,1) \left\|f \right\|_{L^{1}(\mathbb T; w)} 
\mathrm{log}_1 \left (  \frac{\left\|f \right\|_{L^{q}(\mathbb T; w)} }{\left\|f \right\|_{L^{1}(\mathbb T; w)} }   \right), \qquad 1<q\leq \infty,
\\ \label{e:qa}
& \left\| \mathcal{C} f \right\|_{L^{1,\infty}(\mathbb T; w)} \leq C\left\|   f \right\|_{\mathrm{QA}_q(w)}, \qquad 1<q\leq \infty,
\\ & \label{e:log3}
 \left\| \mathcal{C} f \right\|_{L^{1,\infty}(\mathbb T; w)} \leq C \left\|   f \right\|_{L\mathrm{log}_1 L\mathrm{log}_3 L(\mathbb T; w)}.
 \end{align} See  \eqref{e:defQA} for the definition of the $\mathrm{QA}_q(w)$-quasinorms.
  Additionally, as a consequence of  \eqref{e:qa},  the Fourier series of $f\in \mathrm{QA}_q(w)$ converges pointwise a.e.\ whenever $w\in A_1$.
\end{corollary}
\begin{proof}
\noindent Estimate
\eqref{e:weighted1} is obtained by  using   the $(p,1)$-sparse bound of Theorem \ref{t:A} as the input of \cite[Theorem 1.4]{DORONIA}. 
For \eqref{e:mteantonov}, a consequence of \eqref{e:weighted1} is that
\begin{equation}\label{weakl1proof}
	\left\| \mathcal{C} f \right\|_{L^{1,\infty}(\mathbb T; w)}\leq   \left\| \mathcal{C} f \right\|_{L^{p,\infty}(\mathbb T; w) } \leq  \frac{C\mathsf{K}(w,p)}{p-1} \left \| f \right \|_{L^{p}(\mathbb{T};w)} \leq \frac{C\mathsf{K}(w,1)}{p-1} \left \| f \right \|_{L^{p}(\mathbb{T};w)}  \end{equation} holds whenever $1<p\leq 2$.
For $ p<q $,  $
	 \left \| f \right \|_{L^p(\mathbb{T};w)} \leq \left \| f \right \|_{L^1(\mathbb{T};w)}^{1-q'/p'}  \left \| f \right \| ^{q'/p'}_{L^q(\mathbb{T};w)}, $
and  \eqref{e:mteantonov} follows by using  \eqref{weakl1proof} for   $p$ given by the equation $ p'= \max \left \{2, q'\log \left( {\left \| f\right \|_{L^q(\mathbb{T};w)}}/{\left \| f \right \|_{L^1(\mathbb{T};w)} }\right)\right \}  $. Now,
\eqref{e:qa} is deduced from the definition of $\mathrm{QA}_q(w)$ and Kalton's log-convexity of $L^{1,\infty}(\mathbb T;w)$ \cite{Kal}.
 Estimate \eqref{e:qa} immediately implies \eqref{e:log3} once the (strict) continuous 
inclusion
\begin{equation}
\label{e:orl}
 L\mathrm{log}_1 L \mathrm{log}_3L(\mathbb T;w) \subsetneq \mathrm{QA}_\infty(w).
\end{equation}
is established. This is done repeating with obvious changes the argument of \cite[Sect.\ 3.3]{CGMS}.  Note that the inclusion \eqref{e:orl} is tight in the Orlicz class, under modest assumptions on the fundamental function  \cite{CMR}.
\end{proof}
 Another aspect naturally arising in the pursuit of endpoint estimates and pointwise convergence of Fourier series for spaces near $L^1$ is the sharp quantification of the dependence on the weight constants in the weighted bounds for the Carleson operator. For instance, the next result yields that $\mathcal C: L(\mathrm{log}_1L)^2(\mathbb T)\to L^1(\mathbb T) $, via the extrapolation theory of \cite{CDS}. Note that $L(\mathrm{log}_1L)^2(\mathbb T)$ is the  largest Orlicz space currently known to have this property, a result originally due to Sj\"olin \cite{S}.
 \begin{corollary} \label{cor:a2}
The maximal operator \eqref{carleson} obeys the weighted norm inequality
\[
\left\|\mathcal{C}f\right\|_{L^q(w)}\leq C_q[w]_{A_q}^{\frac{\max\{q,2\}}{q-1}}  \left\|f\right\|_{L^q(w)}, \qquad 1<q<\infty.
\]
The same estimates holds for the periodic version of \eqref{e:mainspper}. 
\end{corollary} 
\begin{proof}The claimed  estimate is derived from the sparse bound of Theorem \ref{t:A} as in \cite[Proof of Corollary A.1, (1.5)]{CoCuDPOu}. \end{proof}
\begin{remark}[Comparison with previous results] \label{r:prev} This remark will place our new results in the context of past literature.
 First of all, the Carleson-Hunt estimate \eqref{e:ch} is quantitatively equivalent  to the generalized restricted weak type bound of Lacey and Thiele \cite{LT}, and strictly stronger than the estimate proved by \cite{F}, which, when phrased as a restricted type estimate, is of type  $\mathcal C:L^{p,1} \to L^{q}$  for $1<q<p$. An alternative formulation of
\eqref{e:ch} is 
\begin{equation}
	\label{e:ch2}	\left\|\mathcal{C}f\right\|_{L^{1,\infty}(\mathbb T)} \leq C |F| \mathrm{log}_1\left(\frac{1}{|F|}\right) , \qquad F\subset \mathbb T,\; |f|\leq \cic{1}_F.	\end{equation} 
 Relying on the smoothness of  the Dirichlet kernel via an approximation argument,  Antonov  \cite{Antonov96}   upgraded \eqref{e:ch2} to a mixed type bound which is exactly estimate \eqref{e:mteantonov} with $w=1$ and $q=\infty$, and deduced the $w=1$ case of \eqref{e:log3}. Further work of Sj\"{o}lin and Soria   \cite{SS} extended Antonov's approach to more general sublinear operators satisfying Carleson-Hunt type bounds as in  \eqref{e:ch2}; see also  \cite{GRMARTSORIA}  for applications of this principle to weighted bounds. Arias de Reyna \cite{ADR} introduced the quasi-Banach spaces 
 $\mathrm{QA}_\infty\coloneqq \mathrm{QA}_\infty(\d x)$  and  noticed that Antonov's result may be phrased in terms of \eqref{e:qa} for $w=1$. The observation of \cite{ADR} is relevant because of the strict inclusion \eqref{e:orl}.
The work \cite{Lie13} by Lie gave a  proof of the Lebesgue case of \eqref{e:mteantonov}, with unspecified  dependence on $1<q<\infty$, without any appeal to approximation arguments of the type used in \cite{Antonov96,SS}. In a nutshell, \cite{Lie13} refines the construction of the forests from Fefferman's proof of Carleson's theorem in a $\mathrm{BMO}$ sense. 
The main result of \cite{Lie13} thus implies  the unweighted case of \eqref{e:qa} via the same log-convexity argument. The work  \cite{Lie13} also contains the observation\footnote{In \cite{Lie13}, the observation that $\mathrm{QA}_\infty$ and $\mathrm{QA}_q$ are the same quasi-Banach space is attributed to L.\ Rodriguez-Piazza.} that  $\mathrm{QA}_\infty$ and $\mathrm{QA}_q$ are equivalent   quasinorms  for each $1<q<\infty$, so that the results of \cite{Lie13} and \cite{ADR} are formally equivalent. 

As far as prior weighted bounds at the endpoint $p=1$, the  work of Carro and Domingo-Salazar \cite{CDS} deduces from the Carleson-Hunt bound  \eqref{e:ch} and extrapolation that the Carleson operator maps   $ L \mathrm{\log}_1 L \mathrm{\log}_{3}L(\mathbb{T};w)$ into  the space $R_1(w) $, which is  a  logarithmic correction of $L^{1,\infty}(\mathbb{T};w)$, and the operator norm dependence on $[w]_{A_1}$ is unspecified. In view  of the strict continuous embeddings $ L^{1,\infty}(\mathbb T;w) \hookrightarrow R_1(w)  $ and \eqref{e:orl}, and of the dependence of $\mathsf{K}(w,1)$ on $[w]_{A_1}$,  our estimate \eqref{e:qa} improves on \cite[Theorem 4.5]{CDS}.  

Corollary  \ref{cor:a2}  is an improvement on the previously best known quantitative estimate for the  $L^q(w)$ norms of maximally modulated multipliers, due to Lerner and the first author \cite{DPLer2013}. In particular, the extra $ \mathrm{log}_2  [w]_{A_q}$ term appearing in \cite[Corollary 1.2 (ii)]{DPLer2013} is shown to be unnecessary. 

In summary, the weak-$L^p$ bound of Theorem \ref{cor:carl2}, and \emph{a fortiori} the $(p,1) $-sparse estimate of Theorem \ref{t:A}, sharpen the Carleson-Hunt bound \eqref{e:ch}. Theorem \ref{t:A} also yields upgraded versions of   previous results at $p=1$, which are all consequences of  \eqref{e:ch}.
In particular, Corollary  \ref{cor:weighted}   ensures that the Fourier series of any function in the class  \[
X\coloneqq\left\{f\in L^1(\mathbb T):  f\in \mathrm{QA}_\infty(w) \textrm{ for some }  w\in A_1 \right\}\]
converges  almost everywhere. We stress that the class $X$ is not just formally larger than $\mathrm{QA}_\infty$. For instance, for 
\[
w(x) =\frac{1}{|\log(|x|)|^{\frac12}} \cic{1}_{\R\setminus \{0\}}(x), \qquad  f(x) = \frac{1}{x(\log x)^2  \log\log |\log x|} \cic{1}_{(0,\e^{-\e^{\e}})}(x),    
\]
we have $w\in A_1$,   $f \in L\mathrm{log}_1 L\mathrm{log}_3L (\mathbb T; w)\subset \mathrm{QA}_\infty(w)$, $f \not \in \mathrm{QA}_\infty $.
 \end{remark}

\subsection{Methods, organization and and further results}
The proof of Theorem \ref{t:A} is, in es\-sence, a version of the Lacey-Thiele argument from \cite{LT} for functions outside local $L^2$ that avoids interpolation and the consequent loss of constants.  In Section \ref{s:tiles}, matters are reduced to estimating bilinear forms involving  wave packet coefficients \eqref{e:wpt} associated to a \emph{tile} $P$, namely a Heisenberg uncertainty box in the space-frequency plane, and their modified version \eqref{e:wptA}. The wave packet coefficient \eqref{e:wpt}  is roughly the $L^\infty$ norm of the projection of $f$ to a $O(1)$-dimensional subspace of   functions space-frequency   adapted to $P$. Using the outer $L^p$ framework of Do-Thiele \cite{DT2015}, described in Section \ref{s:outer}, the main steps of the proof become two quantified and localized outer Carleson embedding theorems for the wave packet maps \eqref{e:wpt} and \eqref{e:wptA}. The latter is essentially a localized reformulation of the mass parameter bounds of \cite{LT} and occupies Section \ref{s:modifemb}. The former, Theorem \ref{mainemb},   is substantially new, and is stated  and proved in Section \ref{s:lwpe}. Section \ref{s:pcarl} then contains the short and completely elementary stopping forms argument leading to Theorem \ref{t:A}.

The main novel technical tool behind the proof of Theorem \ref{mainemb} is a smooth space-frequency decomposition of a function $f$ locally in $L^p$, $1<p<2$ induced by a  \emph{forest}, namely a collection of tiles organized into space-frequency \emph{trees}. The decomposition is constructed by expanding $f$ in Gabor series spatially localized on Calder\'on-Zygmund intervals associated to the forest, and selecting a principal part  \eqref{e:pr1} which is locally in $L^2$, albeit with   local norms depending on the counting function of the forest. Multi-frequency decomposition  lemmas of different flavor have been used extensively in the past  literature on modulation invariant singular integrals \cite{DPOu,DPThiele,OT}. The construction used in all these references generates a \emph{good part} via projection on the linear span of $N$ pure frequencies on a spatial interval, initially due to Nazarov, Oberlin and Thiele \cite{NOT2010}, and based on a sleek Hilbert space lemma of Borwein-Erdelyi \cite{BE}. The corresponding remainder term does have vanishing moments with respect to the relevant frequencies, but   its local norms are of the same order of those of the good part, and thus also depend on the counting function. This loss may only be offset by paying an additional price on the 
good part. On the contrary,  the smooth remainder  \eqref{e:pr12} from our decomposition inherits the  much smaller local norms of $f$ and its contribution to \eqref{e:wpt} may be estimated as a pure error term, by careful exploitation of   frequency decay and separation in frequency localization. We expect that our smooth decomposition will find extensive use in further problems involving modulation invariant estimates outside local $L^2$, such as, for instance, uniform estimates for the bilinear Hilbert transform, see \cite{MTT2,OT,TU} for context.

The wave packet coefficients \eqref{e:wpt} also appear in the model sums of the  multiplier operators with singularity along  subspaces of rank one, whose archetypal example is the bilinear Hilbert transform. The first $L^p$-bounds for the  latter operator are due to  Lacey and Thiele \cite{LTbht1,LTbht2}, while    Muscalu, Tao and Thiele address more general multipliers and higher ranks \cite{MTT1}. A systematic qualitative weighted theory for rank one multiplier operators was first obtained by Culiuc, Ou and one of us in \cite{CDPOUBHT}, as a corollary of a family of $\vec p$-sparse bounds. Subsequently, several works have deduced from the sparse bounds of \cite{CDPOUBHT} further qualitative weighted and vector-valued norm inequalities  by developing suitable multilinear  extrapolation theorems, see e.g.\ \cite{CUM2018,LMM+,LMO,NieZ2021}. 
On the other hand, it has proved difficult to deduce \emph{quantitative} weighted estimates, i.e.\ with specified, possibly sharp dependence, from the main result of \cite{CDPOUBHT}, mainly because the constant in the $\vec p$-sparse bounds blows up in an unspecified way when the vector $\vec p$ approaches the extremal points of the range. The wave packet embedding of Theorem \ref{mainemb} may be used to quantify the blow up rate much more precisely, leading to the following improvement of \cite[Theorem 1.3] {CDPOUBHT}.
\begin{theorem}
\label{multipliersparse} Let $\Gamma=\{\xi=(\xi_1,\xi_2,\xi_3)\in\R^3: \xi_1+\xi_2+\xi_3 =0\}$ and $\Gamma'=\mathrm{span} \,\gamma$ be a non-degenerate rank 1  subspace of $ \Gamma$, in the sense that  $\gamma=(\gamma_1,\gamma_2,\gamma_3)$ is a unit vector with $\gamma_j\neq 0$ for all $j=1,2,3$. Let
$
m\in L^\infty(\R^3) \cap \mathcal C^{\infty}(\R^3\setminus \Gamma')
$
be a symbol satisfying the estimates
\begin{equation}
\label{e:distbht}
\sup_{\xi \in  \R^3\setminus \Gamma'}
\left[\mathrm{dist}(\xi, \Gamma')\right]^{|\alpha|}\left|\nabla^{\alpha} m(\xi) \right| \leq 1
\end{equation} for all multi-indices $\alpha$ up to some finite order. Then the form\footnote{The  action of  $\Lambda_m$ on tuples of functions $f_j$ which are merely assumed to belong to $L^\infty_0(\R)$ may be defined by smooth truncation of the integral \eqref{e:BHTmod} near $\Gamma $ and at infinity. The obtained bounds are uniform with respect to the truncation parameter, thus allowing for the limiting argument. This is classical, and we omit details.}
\begin{equation}
\label{e:BHTmod}
\Lambda_m(f_1,f_2,f_3) =\int_{\Gamma} m(\xi) \prod_{j=1}^3 \widehat{f}_j(\xi_j) \, \d \xi
\end{equation}
satisfies the family of $\vec p$-sparse bounds
\begin{equation}
\label{e:BHTrange}
\left\| \Lambda_m \right\|_{\vec p} \lesssim \frac{1}{\eps(\vec p)}, \qquad 1\leq p_1, p_2, p_3 < \infty, \qquad \eps(\vec p) \coloneqq  2 - \sum_{j=1}^2 \frac{1}{\min\{p_j,2\}}>0.
\end{equation}
\end{theorem} The proof of Theorem \ref{multipliersparse} is given in Section \ref{s:pfbht}.  Note that the adjoint forms to the (non-degenerate) bilinear Hilbert transforms with parameter $\beta=\gamma \times (1,1,1)$ correspond to the choices
$m(\xi) = \cic{1}_{(0,\infty)}(\beta \cdot \xi) \psi((1,1,1)\cdot\xi ) $, where $\psi$ is any Schwartz function on $\R$ with $\psi(0)=1$.
We do not detail the consequences in terms of weighted bounds for   $\Lambda_m$,  which may be reconstructed by the interested reader via the extrapolation theorems of  \cite{LMM+,LMO,NieZ2021}. Tracking the constants in those works will   lead to quantitative weighted estimates. This point is transversal to the present article and  will   be expounded elsewhere. In a different direction, Theorem \ref{multipliersparse} yields precise information on the behavior of  $\Lambda_m$ near the extremal \cite[Subsect.\ 2.2]{LaceyMax} pair $L^1(\R)\times L^2(\R)$, fully recovering all results obtained in \cite{DPThiele} and leading to several improvements. One of these is detailed in the following corollary, improving in particular \cite[Theorem 3]{DPThiele}.
\begin{corollary} Let $T_m$ be an  adjoint to $\Lambda_m$ from \eqref{e:BHTmod}. Then 
\[
\begin{split} &
\left\|T_m :L^{\frac{1}{1-\eps}}(\R)\times L^{2}(\R)\to L^{\frac{2}{3-2\eps},\infty}(\R)\right\| \lesssim \frac{1}{\eps}, \qquad 0<\eps<2^{-6}, \\
&\left\|T_m :L^{1,\frac23}\log L^{\frac23}(\R)\times L^{2}(\R)\to \frac{L^{\frac{2}{3},\infty}}{\log L} (\R)  \right\| \lesssim 1.
\end{split}
\]
\end{corollary}
For the definition of the spaces appearing in the second estimate and   its easy deduction from the first, see \cite[Theorem 4]{DDPJFAA}.
\begin{remark}[On the relationship between sparse and weak type] Corollaries \ref{cor:weighted} and \ref{cor:a2} demonstrate how sparse  bounds are both formally stronger    and  convey additional information than Lebesgue estimates.     The article \cite{DPLer2013},   based  on mean oscillation techniques, contains a partial converse of the sparse to weak type implication for maximal modulation singular integrals. The  weighted estimates of \cite{DPLer2013} have in fact been deduced relying on weak-$L^p$ type bounds which are strictly weaker than both  Theorem \ref{t:A} and the Carleson-Hunt bound \eqref{e:ch}, and which are in fact consequences of  \eqref{e:ch}  and extrapolation; see e.g.\ \cite[Estimate (1.7)]{DPLer2013}.   On the other hand, our embedding Theorem \ref{mainemb} yields   Theorem \ref{t:A} directly, and also applies in the context of Theorem \ref{multipliersparse},  which is out of reach for current mean oscillation techniques: see \cite{LOP} for a more recent, unified approach to sparse domination via weak type bounds.  \end{remark}
 
	\subsection*{Recurring notation} The treatment in this paper focuses on the case of functions defined on the real line; however, the generalization to higher dimensional Euclidean spaces is merely notational and all arguments are easily transcribed to that setting. 
	The Fourier transform on $\R$ obeys the normalization
\[
\mathcal F f (\xi)= \widehat f(\xi) =\frac{1}{\sqrt{2\pi}} \int_{\R} f(x) \e^
{-ix\xi} \, \d x, \qquad \xi \in \R. 
\]	
Throughout, the transformations
\[
\mathrm{Tr}_a f\coloneqq f(\cdot -a), \qquad \mathrm{Mod}_a f\coloneqq  \exp(ia\cdot)f(\cdot), \qquad \mathrm{Dil}_b^{p} f\coloneqq b^{-\frac1p}f(b^{-1} \cdot)
\]
for $a\in \R, b>0$, $0<p\leq \infty$,
are used to describe the invariance properties of our singular operators.
The symbol
	\[\langle x \rangle =\sqrt{1+|x|^2}, \qquad x\in \R\] indicates the usual Japanese bracket.	
	The center and length of an interval $I\subset\R$ are respectively denoted by $c_{I} $ and $\ell_I$. Accordingly, define the $L^\infty$-normalized polynomial decay factor adapted to $I$ by
	\[
	\chi^M_I \coloneqq  \mathrm{Tr}_{c_I} \mathrm{Dil}_{\ell_I}^\infty \langle\cdot \rangle^{-M}, \qquad M \in 2\mathbb N \setminus \{0\}.
	\]
	When we drop $M$ and simply write $\chi_I$ instead, the parameter $M$ is large and unimportant.
As customary, for $0<p\leq \infty$,  local $L^p$-(quasi)norms on $I$, their tailed analogues and the $p$-th Hardy-Littlewood maximal operator  follow the notation
\[
\langle f \rangle_{p,I} \coloneqq  |I|^{-\frac1p} \left\| f \cic{1}_I \right\|_{p}, \qquad \llangle  f \rrangle_{p,I} \coloneqq  |I|^{-\frac1p} \left\| f \chi_I^{\frac{2^9}{p}} \right\|_{p}, \qquad \mathrm{M}_p f\coloneqq \sup_{I\subset \R} \langle f \rangle_{p,I} \cic{1}_I.
\]	
with most times $\mathrm{M}_1=\mathrm{M}$ for simplicity. 
We   clarify our notation for the weighted Lorentz and Orlicz spaces appearing in the results of Corollary \ref{cor:weighted}. A \emph{weight} stands for a positive integrable function  $w$ on  $\mathbb T=(-\pi, \pi] $. There is no loss in generality with assuming that $\int_{\mathbb T} w= 1$. As customary, we overload the notation for the weight $w$ and the corresponding measure $\d w =w \d x$.  The weak and strong weighted Lebesgue quasinorms are then defined for $p\in (0,\infty)$ by
\[
\|f\|_{L^{p,\infty}(\mathbb T; w)} \coloneqq \sup_{t>0} t \left[w\left(\{x\in \mathbb T: |f(x)|>t\}\right)\right]^{\frac1p}, \qquad \|f\|_{L^{p}(\mathbb T; w)} \coloneqq \left(\int_{\mathbb  T} |f|^p \, \d w \right)^{\frac1p}.
\]
If $\Phi:[0,\infty]\to [0,\infty)$ is a fundamental function,  the weighted Orlicz norm $\Phi(L)(\mathbb T;w)$ is
\[
\|f\|_{\Phi(L)(\mathbb T;w)} \coloneqq \inf\left\{ t>0: \int_{\mathbb T} \Phi\left( \frac{|f|}{t}\right) \, \d w \leq 1\right\}.
\]
The  fundamental functions occurring are $\Phi(t)=t\mathrm{log}_1 t$ and $\Phi(t)= t\mathrm{log}_1 t\mathrm{log}_3 t$, with iterated logarithm  notation
\[
\mathrm{log}_1 t =\max\{1,\log t\}, \qquad \mathrm{log}_k t =\max\{1,\log(\mathrm{log}_{k-1}t)\}, \quad k\geq 2.
\]
The quasinorm  $\mathrm{QA}_q(w)$ appearing in \eqref{e:qa} is defined by
\begin{equation}
\label{e:defQA}
\|f\|_{\mathrm{QA}_q(w)} = \inf\left\{\sum_{j=1}^\infty  \mathrm{log}_1 j
 \left\|   f_j \right\|_{L^{1}(\mathbb T; w)} \mathrm{log}_1 \left (\textstyle\frac{\left\|f_j \right\|_{L^{q}(\mathbb T; w)} }{\left\|f_j \right\|_{L^{1}(\mathbb T; w)} }   \right): f=\sum_{j=1}^\infty f_j, \sum_{j =1}^{\infty}|f_j|<\infty \text{ a.e.} \right\}.
\end{equation}
 Finally, the symbol $C$ and the  constant implied by the almost inequality sign  $\lesssim$ are meant to be absolute,  unless otherwise specified via the notation $C_{a_1,\ldots, a_n}, \lesssim_{a_1,\ldots, a_n}$. The latter notation highlights dependence on the parameters ${a_1,\ldots, a_n}$.

\subsection*{Acknowledgments} Part of this work has been carried out while the authors were in residence at the program ``Interactions between Geometric measure theory, Singular integrals, and PDE'' held at the Hausdorff Institute for Mathematics, Bonn during Spring 2022. The authors gratefully acknowledge the organizers of the program and the academic and supporting staff of the Institute for the delightful hospitality.

The authors benefited from, and are thankful for, inspiring conversations with Marco Fraccaroli and Christoph Thiele on the $X^{p,q}_a$ outer spaces of Section \ref{s:outer}, and with Kangwei Li on  the topic of multilinear weighted extrapolation. The authors also extend their gratitude to Andrei Lerner for fruitful discussion on sparse domination principles, and  to Maria J.\ Carro for bringing the interesting reference \cite{CDS} to their attention.
\section{Space-frequency analysis of modulation invariant operators} 
\label{s:tiles} After a few preliminaries, this section introduces the wave packet transform \eqref{e:wpt}
on the space-frequency tiles,   its modified version \eqref{e:wptA}, and their role in the discretization of maximally modulated singular integrals.
\subsection{Dyadic grids and tiles} We say that a collection $\mathcal D$ of intervals of $\R$ is  a \emph{dyadic grid} if
\begin{itemize}
\item[a.] $\{\ell_I: I \in \mathcal D\}\subset \rho2^{\mathbb Z}$ for some $\frac12< \rho=\rho_{\mathcal D} <2$.
\item[b.] for all $k\in \mathbb Z$ there holds
$
\mathbb R =\bigcup\{I \in \mathcal D, \ell_I=\rho2^k\}
$
up to possibly a set of zero measure (covering property);
\item[c.] $I,J \in \mathcal D \implies I\cap J\in\{\varnothing, I, J\}$ (grid property).
\end{itemize} The elements of a dyadic grid are referred to as \emph{dyadic intervals}. A typical example that we will use at times are the three shifted dyadic grids
\[
\D_{g}\coloneqq\left\{2^k\left({\textstyle\ell+\frac{g(-1)^k}{3}}+[0,1)\right): k, \ell \in \mathbb Z\right\}, \qquad g=0,1,2.
\]  
\begin{remark}[Parent, sibling, and children intervals] Let $  I \in \mathcal D $ be a dyadic interval and $\kappa\geq 1$. Properties a.\  to  c.\ yield the existence of a unique interval $I^{\mathsf{p}(\kappa)}\in \mathcal D$ with $\ell_{I^{\mathsf{p}(k)}}=2^\kappa \ell_{I}$ and $I\subset I^{\mathsf{p}(\kappa)}$. We call $I^{\mathsf{p}(\kappa)}$ the $\kappa$-th \emph{parent} of $I$.
Conversely, if $I\in \mathcal D$, we enumerate by \[
I^{\mathsf{ch}(\kappa,j)}, \qquad j=1,\ldots, 2^\kappa\] the collection of the  $\kappa$-\emph{grandchildren} of $I$.  These are those $J\in  \mathcal D$ with $J^{\mathsf{p}(\kappa)}=I$, with the obvious convention that  $c(I^{\mathsf{ch}(\kappa,j_1)})<c (I^{\mathsf{ch}(\kappa,j_2)})$ if $1\leq j_1 <j_2\leq 2^\kappa$.
Finally, we denote by $I^{\mathsf{b}}$, the \emph{sibling} of $I$, the unique $J\in\mathcal D $ with $\ell_J=\ell_I$ and $I^{\mathsf{p}(1)}=I\cup J$. 
\end{remark} 
\begin{remark}[Shifted grids]
\label{r:grids2}
Let $M$ be a large integer, Standard shifted dyadic grid techniques, see e.g.\ \cite{LN}, yield the existence of dyadic grids $\mathcal G_{j}$, $j=1,\ldots, 2^{M+10}$ with the following property: for every (not necessarily dyadic) interval $Q\subset \mathbb R$ there exists $j$ and $I(Q)\in \mathcal G_j $ with $Q\subset I(Q)$ and $\ell_{I(Q)}\leq (1+2^{-M}) \ell_Q$. This property will be used a couple of times in what follows.
\end{remark}

We say that the grids $\mathcal D, \mathcal D'$ are \emph{dual} if $\rho_{\mathcal D}\rho_{\mathcal D'}=1$. Let now $\mathcal D \times \mathcal D'$ be a  fixed pair of dual dyadic grids  on $\R$. 
A \emph{tile} $P=I_P\times \omega_P\in  \mathcal D \times \mathcal D'$ is the cartesian product of  dyadic intervals  with reciprocal lengths, that is
$
\ell(I_P)\ell(\omega_P) = 1.
$
The intervals $I_P, \omega_P$ are  referred to respectively as the \emph{spatial support} and  \emph{frequency support} 
of the tile $P$. The set of \emph{all tiles} in $ \mathcal D \times \mathcal D'$ is denoted by $\mathbb S_{\mathcal D, \mathcal D'}$ or simply $\mathbb S$ if the dyadic grids are fixed and clear from context, and referred to as \emph{tiling} associated to $ \mathcal D \times \mathcal D'$ or simply \emph{tiling}.
It is convenient to adopt the notation $\scl(P)=\ell_{I_P}$  for the (spatial)   scale of $P$. 

\subsection{Wave packets and wave packet transforms.} The rationale for defining tiles as above is that they describe the space-frequency localization of the functions, referred to as \emph{wave packets}, involved in the analysis of modulation invariant operators.
Denote by  $\Theta^{M}$   the unit ball of the  Banach space  \[
\left\{\vartheta \in \mathcal C^M(\R): \|\vartheta\|_{\star,M}<\infty, \right\}, \qquad
\|\vartheta\|_{M} \coloneqq \sup_{0\leq \alpha \leq M }\sup_{x\in \R}\left| \langle x\rangle^M D^\alpha\vartheta(x)  \right| 
\] 
For a tile $P=I_P \times \omega_P$, define the corresponding  $L^1$-adapted, localized classes of order $M$ by

\begin{equation}
\label{e:refwps}
\Phi^{M}(P) \coloneqq\left\{ \varphi = \mathrm{Mod}_{ c(\omega_P)} \mathrm{Tr}_{c(I_P)}
\mathrm{Dil}_{\scl(P)}^1 \phi \textrm{ for some }\phi\in \Theta^{M}, \, \supp \widehat{\varphi}\subset \omega_{P}\right\}.
\end{equation} We stress that $\varphi \in \Phi^{M}(P)$ has compact frequency support in $\omega_P$.
From now on, we omit the $M$ from the superscript and our forthcoming definitions depend on $M$ implicitly. The \emph{order $M$ wave packet transform} of  $f\in L^{\infty}_{0}(\R)$ is the map
\begin{equation}
\label{e:wpt}
\begin{split}
 &	W[f]: \mathbb S\to [0,\infty), \qquad W[f](P) \coloneqq \sup_{\varphi\in \Phi^M (P) } \left|\langle f,\varphi\rangle\right|   \end{split}
\end{equation}
The dependence on $M$ is kept implicit in the notation. We can think of $W[f](P)$ as the  magnitude of the space-frequency localization of $f$ to the tile $P$.

  When dealing with maximally modulated singular integrals, a modified wave packet transform   models the contribution of the dualizing function. Namely, define
\begin{equation}
\label{e:wptA}
\begin{split}
 { A}[f](P)\coloneqq \sup_{\psi \in \Psi  ^{M}(P)} \left| \left\langle f, \psi(\cdot,N(\cdot))  {\cic{1}_{ \omega_{P}^{\mathsf{b}}}} (N(\cdot)) \right\rangle \right| 
\end{split}
\end{equation}
 where   $N:\R\to \R $ stands for a fixed measurable function,  and $\Psi ^{M}(P)$ is the modified class
\begin{equation}
\label{e:refwpsO}\Psi^{M}(P)\coloneqq\left\{ \phi=\phi(x,\nu)\in \mathcal C^M(\R\times \R) : \left[\frac{\partial_{\nu}}{\scl(P)} \right]^a \phi(\cdot,\nu) \in \Phi^M(P),  \,\forall \nu \in \R, \, a=0,1  \right\}. \end{equation} 
 The dependence on the function $N$ and on the smoothness-decay parameter $ M$   is kept implicit in the notation for \eqref{e:wpt}-\eqref{e:wptA}, as these  will be clear from context. For this reason, unless strictly necessary,     $M$ is dropped from the notations, writing  for example $\Phi (P), \Psi(P)$.

\subsection{Analysis of maximally modulated singular multipliers \label{maxmodmultimodel}} The wave packet transforms \eqref{e:wpt} and \eqref{e:wptA} enters directly the discrete models of both the Carleson operator and of   rank 1 multilinear multipliers such as the bilinear Hilbert transform. For the sake of motivation, here follows the reduction of the former family to the wave packet form \eqref{e:modelsum} below.
  
Let $m\in L^\infty(\R) \cap \mathcal C^\infty(\R\setminus\{ 0\}) $ be a smooth H\"ormander-Mihlin multiplier, that is
\begin{equation}
\label{e:HM0}
\sup_{0\leq \alpha \leq M} \sup_{\xi \neq 0} |\xi|^\alpha \left|m^{(\alpha)} (\xi) \right| \leq 1
\end{equation}
for some large and unimportant  $M$.
In the next paragraph, we prove the pointwise estimate
\begin{equation}
\label{e:modelred}
\mathcal C f(x) \leq {\sum_{u=1}^{95}}
\sum_{\star\in \{+,-\}} \sum_{g=\{0,1,2\}}  
 \sup_{N\in \R}\left| \sum_{P\in \mathbb S_{g}} |I_P|\langle f, \phi_P\rangle \psi^\star_{P{,u }}(x,N) \right|, \qquad x\in \mathbb R
\end{equation}
where 
\[
\mathcal C f(x)=
\sup_{N \in \R} \left| H_N f \right| , \qquad H_N f(x)\coloneqq \int_\R m(\xi-N)\widehat{f}(\xi)  \e^{i x \xi} \, \d \xi, \quad x \in \R.
\]
is the maximally modulated multiplier operator already introduced in \eqref{carleson}, 
   $\mathbb S_{ g}\coloneqq\mathbb S_{\mathcal D_0, \mathcal D_g}$ is the set of all tiles associated to the grids $\mathcal D_0, \mathcal D_g$,
  the functions  $\phi_P,\psi_{P{,u }}^\star$ are uniform multiples of adapted wave packets from respectively 
$\Phi(P), \Psi(P)$, cf.\ \eqref{e:refwps}-\eqref{e:refwpsO}, 
and
\begin{equation}
\label{e:supppsi1}
\mathrm{supp}_2\psi^{\pm}_{P{,u }} \coloneqq \overline{\left\{N\in \R: \psi^{\pm}_{P{,u }}(\cdot,N)\neq 0\right\}} \subset Q_P^{\pm,u}\coloneqq c(\omega_P)\mp\ell_{\omega_P} \left[\textstyle  7+\frac{u}{4}  , 9+\frac{u}{4}  \right] \end{equation} 
Fix the parameters $g,u$ and $\star=+\in\{+,-\}$ . We claim that there exist dyadic grids $  \mathcal{G}_j, \; j=1,\ldots ,2^{18} $ with the property that for all $P\in \mathbb{S}_g $ there exists $j_P\in \{1,\ldots ,2^{18}\}$ and $J_P \in  \mathcal{G}_{j_P}$ with 
\[\mathrm{supp}_2\psi^{+}_{P{,u }} \subset
J_P^{\ch(1,1)}, \qquad \omega_P\subset J_P^{\ch(1,2)}.
\]
This is easily obtained by applying Remark \ref{r:grids2} with $ M=8 $ to the convex hull of  $Q_P^{+,u}$ and $\omega_P$, whose  leftmost fourth contains $Q_P^{+,u}$ and is contained within the left half of the smoothing interval, and whose   rightmost fourth contains $\omega_P$, and is contained within the right half of the smoothing interval. We then define the tile $ \tilde{P}=\tilde{P}(P)=I_{\tilde{P}} \times \omega_{\tilde P} \in \mathbb{S}_{\mathcal{H}_j \times \mathcal{G}_j} $ by \[I_{\tilde P} \coloneqq \text{ the unique } J \in \mathcal{H}_j \text{ with } c(I_P)\in J, \ell_J \ell_{J_P^{\ch(1,2)}}=1, \qquad \omega_{\tilde{P}} \coloneqq J_P^{\ch(1,2)}\]
where $\mathcal H_j $ is a fixed dual grid to $\mathcal G_j $.  With this definition,
\begin{equation}
	\label{e:supppsi2} \mathrm{supp}_2\psi^{+}_{P{,u }}\in \omega_{\tilde P(P)}^{\mathsf{b}}.
\end{equation}
For $\tilde P\in  \mathbb S_{\mathcal H_j \times \mathcal G_j }$, let $\mathbb S_g(\tilde P)\coloneqq\{P\in \mathbb S_g: \tilde P(P)= \tilde P\}$. For each $N\in \mathbb R$, we then have
\[
\sum_{  P \in \mathbb S_g(\tilde P)} |I_P|\langle f, \phi_P\rangle \psi^+_{P,u}(x,N) =  \sum_{  P \in  \mathbb S_g(\tilde P)} |I_P|\langle f, \phi_P\rangle \psi^+_{P,u}(x,N) \cic 1_{\omega_{\tilde P}^{\mathsf{b}}} (N)
\]
having used \eqref{e:supppsi2} in the first equality.
By construction, it is then easily verified that 
\[
\#\mathbb S_g(\tilde P) \lesssim 1 , \qquad
\phi_P \in C\Phi(\tilde P), \;   \psi^+_{P,u} \in C\Psi(\tilde P) \quad \forall P \in \mathbb S_g(\tilde P) \] with uniform constants over $\tilde P \in \mathbb S_{\mathcal H_j \times \mathcal G_j }.$
Linearization of the suprema in \eqref{e:modelred}, a passage to the adjoint followed by using the  definitions of \eqref{e:wpt}, \eqref{e:wptA},  and a limiting argument thus  allow us to reduce estimation of the operator \eqref{carleson} to proving uniform bounds for the forms
\begin{equation}
\label{e:modelsum} \mathsf C_{\mathbb P}(f_1,f_2)\coloneqq \sum_{P\in \mathbb P}  {|I_P|} W[f_1](P)A[f_2](P) 
\end{equation} 
where $\mathbb P$ is a finite subset of   $\mathbb S=\mathbb{S}_{\mathcal D,\mathcal D'}$ for a fixed pair of dual grids $\mathcal D,\mathcal D'$, and the function $N(\cdot)$ in the definition \eqref{e:wptA} of $A[f_2](\cdot)$ is a fixed but arbitrary measurable function.

\begin{proof}
[Proof of estimate \eqref{e:modelred}] By splitting and symmetry, we may assume that $m$ is supported on the positive half-line, and obtain the $\star=+$ term, whose superscript is omitted throughout. Let $  \psi_u \in \mathcal S(\R) $, $u=1,\ldots, 95$ with  \[\supp \widehat{\psi_u} \subseteq {\left[\textstyle\frac{1}{2}+\frac{u-1}{64},\frac{1}{2}+\frac{u+1}{64}\right]},\qquad \sum_{u=1}^{95} \sum_{k \in \Z} \widehat{\psi_{u}}(2^k \xi)=\cic{1}_{\left(0,\infty\right)}(\xi) \] and perform the corresponding Littlewood-Paley decomposition of the multiplier   $H_0 f$  as
 \[H_0 f
 ={\sum_{u=1}^{95}}\sum_{k\in \mathbb Z} f*\Psi_{k,u} , \qquad \Psi_{k,u} (x)\coloneqq \int_{\R} m(\xi) \widehat{\psi_{{u}}}(2^k\xi) \e^{ix\xi}\, \d \xi, \quad x\in \R. 
 \] 
 Further, let $\mathcal D_g$, $g=0,1,2$ be the three $1/3$-shifted dyadic grids on $\mathbb R$, and $
 \mathbb S_g(k)=\{P\in \mathbb S_g: \scl(P)= 2^{k} \}
$
 be the corresponding scale $k$ tiles.
Performing the standard Gabor decomposition, we pick $ \phi \in \mathcal S(\R) $ with $  \supp\widehat{\phi} \subseteq \left[0,\frac{2}{3}\right]$ such that \[ \sum_{\lambda \in \Z} \left|\widehat{\phi}\left(\xi-\textstyle\frac{\lambda}{3}\right)\right|^2=1, \qquad \xi \in \R\]  so that  for each $k\in \mathbb Z$
 \[\begin{split}
 &	f= \sum_{g=0}^2 \sum_{P\in  \mathbb S_g(k) }|I_P| \l f , \phi_{P} \r \phi_{P} , \qquad  \phi_{P}\coloneqq \mathrm{Mod}_{c_{\omega_P}}  \mathrm{Tr}_{c_{I_P}}   \mathrm{Dil}^1_{{\scl(P)}} \phi  \end{split}\]  holds. Note that $\phi_P\in C_{a} \Phi^{a}(P) $ for all $a$. 
 Combining and using the frequency support property of $\psi_k$  to restrict the summation,  \begin{equation}
\label{e:redmod01}
\begin{split} & \quad 
 H_Nf= {\sum_{u=1}^{95}} \sum_{k \in \Z}  f* \mathrm{Mod}_N  \Psi_{k,u} ={\sum_{u=1}^{95}} \sum_{g=0}^2\sum_{k\in \mathbb Z} \sum_{P\in  \mathbb S_g(k+4) } |I_P| \l f , \phi_{P} \r \phi_{P} * \mathrm{Mod}_N  \Psi_{k,u} \\ &={\sum_{u=1}^{95}}\sum_{g=0}^2   \sum_{P\in  \mathbb S_g  }  
|I_P| \l f , \phi_{P} \r  \psi_{P,{u }}(\cdot, N)\cic{1}_{{\left [ 7+\frac{u}{4}, 9+\frac{u}{4} \right ]}}\left( \textstyle\frac{c(\omega_P)-N}{\ell_{\omega_P}}\right) 
 \end{split}
   \end{equation}
   having defined the functions
\begin{equation}
\label{e:psiP}
   \psi_{P{,u }}(x, N)\coloneqq\phi_{P} * \mathrm{Mod}_N  \Psi_{k{,u }} (x),  \qquad 2^{k+4} = \scl(P).
\end{equation}
   To obtain \eqref{e:redmod01}, we have used that $\supp \widehat{ \mathrm{Mod}_N  \Psi_{k{,u} }}\subset N+2^{-k}[\frac{1}{2}+\frac{u-1}{64}, \frac{1}{2}+\frac{u+1}{64}]$ and that when $P\in \mathbb S_g(k+4)  $ the frequency support of $\phi_P$ is an interval $\omega_P$ of length $2^{-k-4}$. Thus, in order for $\psi_{P,{u }}(\cdot,N)$ to be  nonzero, $N$ must belong to the interval $ Q_P^{+,u}$ as claimed in \eqref{e:supppsi1}.

    We are left with proving that $\psi_{P,u}\in C_a\Psi^a(P)$ for all $0\leq a\leq M-1$. To this aim, fix $P\in \mathbb S_g(k+4)$. We treat both cases $\alpha=0,1$ at the same time. First of all, using the H\"ormander-Mihlin condition \eqref{e:HM0}
\begin{equation}
\label{e:HM1}
\supp \widehat{\Psi_{k{,u }}} \subset [2^{-k-1},2^{-k+1}], \qquad \left| D^a\widehat{\Psi_{k{,u }}} (\xi)\right| \lesssim_{M} 2^{ka}\sim_M |\xi|^{-a} 
\end{equation} 
for all $0\leq a \leq M$, $k\in \mathbb Z$.  Let also $\beta$ be an auxiliary Schwartz function with the property that $1_{[-\frac{1}{2},\frac12]}\leq   \beta\leq 1_{[-1,1]}$  and define
\[
\Phi_{k,{u, }P}(x,N) \coloneqq \int_{\R} \left({-\textstyle \frac{D}{\scl(P)} }\right)^\alpha \widehat{\Psi_{k{,u }}} (\xi) \beta\left(  \frac{ \xi +N-c_{\omega_P} }
{\ell_{\omega_P}}\right) \, \e^{ix\xi} \, \frac{\d \xi}{\sqrt{2\pi}}, \qquad x\in \R.
\]
Using the Fourier transform and the definition, we check that 
\[
\left[\frac{\partial_{N}}{\scl(P)}\right]^\alpha  \psi_{P{,u }}(\cdot, N)  = \phi_P * \mathrm{Mod}_N \Phi_{k{,u },P}
\]
so our claim follows easily from the scale $\scl(P)\sim 2^{k}$ bump function estimates for the function $\Xi=\mathrm{Mod}_{N-c_{\omega_P}} \Phi_{k,u,P}$, whose Fourier transform is supported on $|\xi|\leq  \ell_{\omega_P}\sim 2^{-k} $ and satisfies
\[
\begin{split}
 \widehat{\Xi}(\xi)&= \left({-\textstyle \frac{D}{\scl(P)} }\right)^\alpha \widehat\Psi_{k{,u }} \big(\xi-(N-c_{\omega_P})\big) \beta\left(  \frac{ \xi }{\ell_{\omega_P}}\right), \\
\left|\widehat{\Xi}^{(a)}(\xi)\right|&\lesssim_a \sum_{b+c=a} 2^{-k\alpha} \left|\widehat\Psi_{k{,u }}^{(b)} \big(\xi-(N-c_{\omega_P})\big)\right| 2^{kc} \lesssim_a
2^{-k\alpha} \left| \xi-(N-c_{\omega_P})\right|^{-b} 2^{kc} \lesssim_a 1
\end{split}
\]
for $0\leq a\leq M-1$, having used that   $c_{\omega_P}-N\geq 3 \ell_{\omega_P} $, while $|\xi| \leq \ell_{\omega_P}$ on the support of $\beta(\cdot/ \ell_{\omega_P})$. This completes  the proof of \eqref{e:modelred}.
 \end{proof}
 \section{Outer $L^p$ estimates for the wave packet transforms} \label{s:outer}
Outer $L^p$ spaces, introduced in this context  by Do and Thiele \cite{DT2015}, provide the functional setting   for our estimates on the wave packet transforms. In this section, after particularizing the main definitions, we introduce two new outer $L^p$ norms  enjoying a weaker, but more precisely quantified form of the outer H\"older inequality.
In what follows, we refer to a fixed tiling $\mathbb S=\mathbb S_{\mathcal D,\mathcal D'}$.
\subsection{Trees} Let $\kappa $ be a nonnegative integer. We say that $T\subset \mathbb S$ is a $ \kappa $-tree if there exists an interval $I_T\in \mathcal D$ and a frequency $\xi_T\in \R$ such that
\[
I_P\subset I_T,\quad \xi_T\in  \omega_P^{\mathsf{p}(\kappa)} \qquad \forall P\in T.
\]
The pair $(I_T,\xi_T)$ is referred to as \emph{top data} of $T$. 
The notation 
 \[ \mathcal{I}(T)\coloneqq  \{  I \in D: I=I_P \text{ for some } P \in T\}, \qquad 	\Omega(T) \coloneqq \{\omega\in \mathcal D': \omega=\omega_P \textrm{ for some } P\in T\} \] is used for the spatial and frequency components of a ${\kappa}$-tree $T$. 

Let $1\leq j\leq2^\kappa$. We say that a $\kappa$-tree $T$ is of type $j$ if $\omega_P=[\omega_P ^{\mathsf{p}(\kappa)}]^{\mathsf{ch}(\kappa,j)}$, that is, equals  the 
 $j$-th $\kappa$-grandchild of its $\kappa$-parent, for all $P\in T$. Clearly any $\kappa$-tree $T$   splits as the disjoint union  $T =\bigsqcup_{j=1}^{2^\kappa} T_{|j}$,  with each $T_{|j}$ being  a $\kappa$-tree of type $j$ with the same top data.
 \begin{remark} \label{singlescalesingletime}
The structure of $\mathbb S$ and the above definition entails that 
the intervals $\{\omega^{\mathsf{p}(\kappa)}: \omega\in \Omega(T)\}$ are nested. Therefore, 
$
\#\{\omega \in \Omega (T): \ell_{\omega}= \rho \} \leq 2^{\kappa}
$  for all $\rho>0$. As a first consequence, 
\begin{equation}
\label{e:spectree} \#\{P\in T:I_P=I\}\leq 2^\kappa\qquad \forall I \in \mathcal I(T).
\end{equation}

 \end{remark}		 
In general, each tree $T$ contains both a Littlewood-Paley type and a maximal function type component. The next definition isolates the Littlewood-Paley part. Say that a $\kappa$-tree $T$ is \emph{lacunary} if
\begin{equation} 
\label{e:lacdef}
	\omega,\omega'\in\Omega(T), \omega\neq \omega'\implies \omega \cap \omega'=\varnothing.
 \end{equation}
and for every tree $T$,  split
\begin{equation}
\label{e:lacov}
T=T^{\mathsf{ov}} \cup T^{\mathsf{lac}}, \qquad T^{\mathsf{ov}}\coloneqq\{P\in T: \xi_T\in \omega_P\}, \qquad T^{\mathsf{lac}}\coloneqq T\setminus T^{\mathsf{ov}}.
\end{equation}
The next lemma tells us in particular that $T^{\mathsf{lac}}$ is a union of at most $\kappa2^{\kappa}$ lacunary trees, and that the residual part  $T^{\mathsf{ov}}$ has additional structure. 
\begin{lemma}[Structure of trees]
	\label{lacpartislac}  Let $ T$ be a $\kappa$-tree   with top data $ (I_T,\xi_T) $.
Then $T=\bigsqcup_{u=1}^{\kappa} T^{u}$, with each $T^{u}$ also a $\kappa$-tree with the same top data and such that, for all $j=1,\ldots, 2^\kappa$
\begin{itemize}
\item[\emph{(i)}] $ \big[T^u_{|j}\big]^{\mathsf{lac}}$  is a lacunary tree;
\item[\emph{(ii)}] whenever  $j' \neq j$,   the intervals $\big\{ \big[\omega_P ^{\mathsf{p}(\kappa)}\big]^{\mathsf{ch}(\kappa,j')}:P\in\big[T^u_{|j}\big]^{\mathsf{ov}}\big\}$ are pairwise disjoint.
\end{itemize}
\end{lemma}
\begin{proof} 	Immediately verified by setting $ T^u  \coloneqq \left\{P \in T    : \scl(P) \in 2^{\kappa \Z+u} \right\} $ for $ 1\leq u \leq \kappa$ . \end{proof}

\subsection{Outer $L^p$  on the space of local  tiles} For $J\in \mathcal D$, let $\mathbb S^J$ be the collection of all tiles $P\in \mathbb S$ with $ {I_P \subset J}$. 
Below, the notation $\ell^p(\mathbb S^J)$ stands for the $\ell^p$ spaces on $ \mathbb S^J $endowed with the weighted counting measure
\[
A\mapsto \sum_{P\in A} |I_P|, \qquad A\subset \mathbb S^J.
\] 
The collection $\mathcal T^{J,\kappa}$ of  all $\kappa$-trees $T\subset \mathbb S^J$ concurs to the definition of the  outer measure space   $(\mathbb S^J,\mathcal T^{J,\kappa},{\mu^{J,\kappa}})$, with outer measure ${\mu^{J,\kappa}}$   defined  by
\begin{equation}
\label{e:outerLp}
{\mu^{J,\kappa}}:\mathcal P( \mathbb S^J)\to [0,\infty], \qquad
{\mu^{J,\kappa}}(A)\coloneqq \inf\left\{ \frac{1}{|J|} \sum_{T\in \mathcal T} |I_T|:  \mathcal T\subset \mathcal T^{J,\kappa},  A\subset \bigcup_{T\in \mathcal T}  T   \right\};
\end{equation} to wit,
the  infimum above  is  taken over all collections $ \mathcal T\subset \mathcal T^{J,\kappa}$ of $\kappa$-trees whose union covers $A$.   Below, for a quasi-subadditive size map $\mathsf s$ as defined in \cite[Def.\ 2.3]{DT2015},
\[
\big[F: \mathbb S^J \to \mathbb C \big]\mapsto \left\{\mathsf{s}(F,T): T\in {\mathcal T^{J,\kappa}}\right\}
\]
we  consider the outer $L^{p,r}$ space on $(\mathbb S^J,{\mathcal T^{J,\kappa}},{\mu^{J,\kappa}})$, \[
L^{p,r}(J,\kappa,\mathsf{s})= L^{p,r	}(\mathbb S^J, \mu^{J,\kappa}, \mathsf{s})\] as defined in \cite[Def.\ 3.2]{DT2015}, for exponents $1\leq p,q\leq \infty$. The definition therein may be summarized as follows. First of all, define the outer essential supremum
\[
\outsup_\s F\coloneqq \sup_{T\in {\mathcal T^{J,\kappa}}} \s(F, T)  \eqqcolon
\|F\|_{L^{\infty}(\s)}= \|F\|_{L^{\infty,\infty}(\s)} 
\]
Secondly, define the super level measure $\mu_{\s}[F] :[0,\infty)\to [0,\infty]$ and the corresponding nondecreasing rearrangement $F^{*,\s}:[0,\infty)\to [0,\infty]$ respectively by
\[\begin{split}&\mu_{\s}[F](\tau)
	\coloneqq  \inf\left\{ {\mu^{J,\kappa}}(A):\outsup_\s (F\cic{1}_{\mathbb S^J\setminus A})\leq \tau\right\},  \\
	&F^{*,\s}(t)
	\coloneqq  \inf\left\{  \tau\in[0,\infty): \mu_{\s}[F](\tau)\leq t\right\}.
\end{split}\]
We then set 
\begin{equation}
\label{e:outerLpsp}
\|F\|_{L^{p,r}(J,\kappa,\mathsf{s})} \coloneqq \left\|F^{*,\s} \right\|_{p,q} =  
\left\| t^{\frac1p}  F^{*,\s}(t) \right\|_{L^q\left([0,\infty), \frac{\d t}{t} \right) };
\end{equation}
recall that the right hand side is  the standard Lorentz $L^{p,r}$ quasinorm on $[0,\infty]$, see e.g. \cite[Sect.\ 1.4]{GrafBook1}.
As customary, when $q=p$ we omit $q$ from the subscripts and superscripts. 

The main  examples of sizes and associated   outer $L^{p,r}$ spaces that arise in our applications are the following.  For $1\leq p \leq \infty$, set
\[ \mathsf{size}_{p}(F, T)\coloneqq  \frac{\left\|F\cic{1}_T\right\|_{\ell^p(\mathbb S^J)}}{|I_T|^{\frac1p}}, \qquad T\in {\mathcal T^{J,\kappa}}.
\] 
For $p=2$, we  define the variant 
\begin{equation}
\label{e:sizestar}
\size_{2,\star}(F, T)\coloneqq \sup\left\{\mathsf{size}_{2}(F, U): U\in {\mathcal T^{J,\kappa}} \textrm{ lacunary},  U \subset T \right\}, \qquad T \in {{\mathcal T^{J,\kappa}}},
\end{equation}
which is also a   size.
The definition of   $\mathsf{size}_{p}(F, \cdot)$ and $\mathsf{size}_{2,\star}(F, \cdot)$ depends  on ${\kappa}$ via the domain $\mathcal T^{J,\kappa}$, though we do not keep this dependence explicit in the notation.

The modified wave packet transform    acting on the dual side of the  Carleson operator, in accordance with  the definition  \eqref{e:wptA} involving $\omega_P^{\mathsf{b}}$, will be estimated in outer $L^{p,r}$-spaces \eqref{e:outerLpsp} where the parameter $\kappa$ is naturally chosen to be 1. On the outer measure space   \eqref{e:outerLp} we thus define, with reference to\eqref{e:lacov} 
\begin{equation}
\label{e:sized}
 \size_{\mathsf{C}}(F,T)\coloneqq\size_{2}(F,T^{\mathsf{lac}})+\size_1(F,T^{\mathsf{ov}}), \qquad T \in {{\mathcal T^{J,1}}}.
\end{equation}

The next proposition is a generalization to the Lorentz scale of the outer H\"older inequality, which plays a pivotal r\^ole in the applications of outer spaces to modulation invariant singular integrals.
\begin{proposition} 
\label{p:holder}
Let $  m \in \N_{\geq 2} $ and  $\s,\s_1,\s_2,\cdots,\s_m$ be  sizes on $(\mathbb S^J,{\mu^{J,\kappa}},{{\mathcal T^{J,\kappa}}})$ with the property that for all function $ m $-tuples $F_1, \cdots ,F_m: \mathbb S^J\to \mathbb C$, 
	\begin{equation}\label{e:submulti}
	\s\left(\prod_{j=1}^m F_j,T\right) \leq  \prod_{j=1}^m  \s_j(F_j,T)   \qquad \forall T \in {{\mathcal T^{J,\kappa}}}.
	\end{equation}
Then for all tuples $0<p,p_1,\ldots, p_m,\, q,q_1,\ldots, q_m\leq \infty, \, \frac{1}{p}= \sum_{j=1}^m \frac{1}{p_j}, \,\frac{1}{q}= \sum_{j=1}^m \frac{1}{q_j}$ there holds
\[
	\left\| \prod_{j=1}^m F_j \right\|_{L^{p ,q}{(J,\kappa,\s_\ell)}} \leq m^{\frac1p}   \prod_{j=1}^{m}\left\| F_j \right\|_{L^{p_j,q_j}(J,\mathsf{s}_j)}.
\]

\end{proposition}
\begin{proof} Chasing definitions, it is immediate to see that
\[
\left[ \prod_{j=1}^m F_j\right]^{*,\s}  \left(\frac t m\right) \leq   \prod_{j=1}^m F_j ^{*,\s_k}  \left(t\right), \qquad 0<t<\infty
\]
and the claim follows from the usual H\"older inequality on the spaces ${L^{q_j}\left([0,\infty), \frac{\d t}{t} \right) }$.
\end{proof}
\begin{remark} Let the assumptions of Proposition \ref{p:holder} stand, and particularize to the case $\mathsf s=\size_1$ and $p=q=1$. Then,
\begin{equation}\label{usualholder}
	\frac{1}{|J|}\left\| \prod_{j=1}^m F_j \right\|_{\ell^1(\mathbb{S}^J)}\lesssim 
		 \left\| \prod_{j=1}^m F_j\right\|_{L^{1}(J,\kappa,\mathsf{size}_1)} \leq m\prod_{j=1}^{m}\left\| F_j \right\|_{L^{p_j}(J,\kappa,\mathsf{s}_j)} 
\end{equation}
where   \cite[Prop.\ 3.6]{DT2015} has been used to get the first bound.\end{remark}
\begin{remark} \label{weaklpineq} 
Let $A\subset \mathbb S^J$ be a set of finite outer measure ${\mu^{J,\kappa}}$. It may be checked directly  that $[\cic{1}_A]^{*,\size_\infty}= \cic{1}_{[0,{\mu^{J,\kappa}}(A))}$, so that in particular $\left\|  \cic{1}_A \right\|_{L^{p,\infty}({J,\kappa,\size}_\infty)} =  {\mu^{J,\kappa}}(A)^{\frac{1}{p}}$. Using monotonicity of the size $\s$, a particular case of Proposition \ref{p:holder} is 
\begin{equation}
\label{e:usefulp} 
\left\| F \cic{1}_A \right\|_{L^{p,q}{(J,\kappa,\s)}} \leq    2^{\frac1p}  \left\| F \cic{1}_A \right\|_{L^{p_1,q}{(J,\kappa,\s)}}  {\mu^{J,\kappa}}(A)^{\frac{1}{p}-\frac{1}{p_1}}, \qquad 0<p\leq p_1\leq \infty, \, 0<q\leq \infty.
\end{equation}
\end{remark}

	\subsection{Reverse H\"older outer $L^p$ norms} The next definition is inspired by Remark \ref{weaklpineq}. Let $\mathsf s$ be any size on $(\mathbb S^J,{\mu^{J,\kappa}},{{\mathcal T^{J,\kappa}}})$, cf.\ \cite[Def.\ 2.3]{DT2015}. 
Define, for $F:\mathbb S^J \to \mathbb C$, $1\leq a \leq p < \infty$, $1\leq q\leq \infty$ and  $\eps>0$, the quasi-norms
\[
\left\|F\right\|_{X^{p,q}_a{(J,\kappa,\s)}} \coloneqq \sup_{A\subset \mathbb S^J} \frac{\left\|F\cic{1}_A\right\|_{{L^{a,q}{(J,\kappa,\s)}}}}{ {\mu^{J,\kappa}}(A)^{\frac{1}{a}-\frac{1}{p}} },\qquad  \left\| F \right\|_{Y^{p,q}{(J,\kappa,\s)}}\coloneqq\max \left\{ \left\| F \right\|_{L^{p,q}{(J,\kappa,\s)}}, \left\| F \right\|_{L^{\infty}{(J,\kappa,\s)}}  \right\}.  
\] 
Remark  \ref{weaklpineq} tells us immediately that
$\|F\|_{X^{p,q}_a{(J,\kappa,\s)}}\leq 2^{\frac1p} \|F\|_{L^{p,q}{(J,\kappa,\s)}}
$ in the range of the definition.
The next proposition  should be interpreted as a partial converse of this control and as a substitute for Proposition \ref{p:holder} with a smaller right hand side.  The $Y^{q,\infty}{(J,\kappa,\s)}$-norm is the quantity appearing in our applications. Formally stronger versions of the proposition, where the $Y$-type norms are replaced by  suitable geometric averages of the $L^{q,\infty}{(J,\kappa,\s)}, L^{\infty}{(J,\kappa,\s)}$, are also available, but we do not detail them. 
\begin{proposition} \label{p:eHolder} Let $  m \in \N_{\geq 2} $ and  $\s_1,\s_2,\cdots,\s_m$ be $ m $ sizes on $(\mathbb S^J,{\mu^{J,\kappa}},{{\mathcal T^{J,\kappa}}})$ with the property that \eqref{e:submulti} holds with $\s=\size_1$.
	Suppose that \[
	1<a\leq p_1<\infty,  \qquad  1\leq p_2,\ldots, p_m < \infty,  \qquad  \eps \coloneqq \left({\sum_{\ell=1}^m} \textstyle\frac{1}{p_\ell} \right)-1  >0. \] Then, with    implicit constant possibly depending on $m$ only, there holds
	\[
	\frac{1}{|J|}\left\| \prod_{j=1}^m F_j \right\|_{\ell^1(\mathbb{S}^J)}	  \lesssim   \frac{  a}{\eps \left(a-1\right)} \left\| F_1 \right\|_{X^{p_1,\infty}_{a}{(J,\kappa,\s_1)}} \prod_{\ell=2}^{m}\left\| F_\ell \right\|_{Y^{p_\ell,\infty}{(J,\kappa,\s_\ell)}} 
	\]

\end{proposition}
\begin{proof}[Proof of Proposition \ref{p:eHolder}]
 Throughout the proof, the constant implied by $\lesssim $ is allowed to depend on $m$ only and vary at each occurrence By scaling we can   assume  \[\left\| F_1 \right\|_{{X^{p_1,\infty}_{a}{(J,\kappa,\s_1)}}
}=
\left\| F_2 \right\|_{Y^{p_2,\infty}{(J,\kappa,\s_2)}}=  \cdots=\left\| F_m \right\|_{Y^{p_m,\infty}{(J,\kappa,\s_m)}}=1.
\] Under this assumption, we must prove
\begin{equation}  
	\label{e:kj0}
	\frac{1}{|J|}\sum_{P\in \mathbb S^J}|I_P| |F_1F_2 F_3  \cdots F_m(P)| \lesssim    \frac{  a}{\eps \left(a-1\right)} .
\end{equation}
Relying on the controls $\left\|F_\ell \right\|_{L^{p_\ell,\infty}{(J,\kappa,\s_\ell)}}, \left\|F_\ell \right\|_{L^{ \infty}{(J,\kappa,\s_\ell)}} \leq 1$ for all $2\leq \ell \leq m$, we iteratively decompose the support of  $F_2F_3 \cdots F_m$ into pairwise disjoint sets $A_{j}$, $j\in \mathbb N$  such that\begin{equation}  
	\label{e:kj1}
	{\mu^{J,\kappa}}(A_{j})   \leq   2^{j}, \qquad 	 \max_{2\leq \ell \leq m} 2^{\frac{j}{p_{\ell}}}{\outsup_{\s_{\ell}}} (F_{\ell}\cic{1}_{A_{j}}) \lesssim   1.\end{equation}
For  $j\in  \mathbb{N}$, let  $k(j)$ be the largest integer $k$ with $k\leq   \frac{ a j}{p_1}$. From  the first estimate in \eqref{e:kj1} and the definition of $X^{p_1,\infty}_{a}{(J,\kappa,\s_1)}$-norm, we learn that  \[\|F_1\cic{1}_{A_j}\|_{L^{a,\infty}{(J,\kappa,\s_1)}}\leq  2^{j\left(\frac1a-\frac{1}{p_1}\right)}.\] Thus, we may further decompose $A_{j}$ into pairwise disjoint sets $\{B_{j,k}:-N\leq k\leq k(j)\}$, where $N$ is an unimportant parameter related to the outer essential supremum of $F_1$,  with
\begin{equation}  
	\label{e:kj2}
	\outsup_{\s_1} (F_1\cic{1}_{B_{j,k}}) \leq 2^{-\frac{k}{a}}, \qquad    {\mu^{J,\kappa}}(B_{j,k}) \leq 2^{k+j\left(1-\frac{a}{p_1}\right)},
\end{equation}
which means that we may find $\mathcal T_{j,k}\subset {\mathcal T^{J,\kappa}}$ with the property
\begin{equation} 
	\label{e:kj3}
	B_{j,k}\subset \bigcup_{T\in\mathcal T_{j,k} } T, \qquad \sum_{T\in \mathcal T_{j,k}} \frac{|I_T|}{|J|} \leq 2{\mu^{J,\kappa}}(B_{j,k}) \lesssim  2^{k+j\left(1-\frac{a}{p_1}\right)}  .
\end{equation}
We then estimate, using  \eqref{e:kj1}, \eqref{e:kj2} and  \eqref{e:kj3} and subsequently summing in $k$, 
\begin{equation} \begin{split} 
		\label{e:kj} & \quad 	\frac{1}{|J|}\sum_{P\in A_  j }|I_P| |F_1F_2 F_3 \cdots F_m(P)|\leq \sum_{ -N \leq k \leq k(j)}  \sum_{T\in\mathcal T_{j,k} } |I_T|  \size_1(F_1F_2 F_3 \cdots F_m \cic{1}_{B_{j,k}}, T) 
		\\
		 & \lesssim   \sum_{ -N \leq k \leq k(j)} 
		 \sum_{T\in\mathcal T_{j,k} } |I_T|  \s_1(F_1 \cic{1}_{B_{j,k}}, T)\prod_{\ell=2}^m\s_\ell(F_\ell \cic{1}_{A_j}, T)
		 \leq   2^{j\left(1-\frac{a}{p_1} - \sum_{\ell=2}^m \frac{1}{p_\ell}\right)}  \sum_{ -N \leq k \leq k(j)} 2^{\frac{a-1}{a} k} 
		 \\ & \lesssim\frac{  a}{a-1}   2^{j\left(1-\frac{a}{p_1} - \sum_{\ell=2}^m \frac{1}{p_\ell}\right)}    2^{\frac{a-1}{a} k(j)}  \lesssim \frac{ a}{a-1} 2^{j\left(1-  \sum_{\ell=1}^m \frac{1}{p_\ell}\right)} = \frac{  a}{a-1} 2^{-\eps j }.
		\end{split}
		\end{equation} 
		The claimed bound \eqref{e:kj0}   follows by summing  the   estimate of the last display over $j\in \mathbb N$.
\end{proof}

\subsection{Lacunary tree estimates} This  paragraph contains some  $\mathrm{size}_{2,\star,{\kappa}}$ estimates for $W[f]$  restricted  to lacunary trees, which we use to explain the role played by this type of trees, and that  will also be of use later. 

Throughout our first discussion, let $T $ be a lacunary tree with top data $(I_T,\xi_T)$. For simplicity, we assume $\xi_T=0$, as the general case of our observations can be recovered by suitably pre- and post-composing with $\mathrm{Mod}_{\pm\xi_T}$.
Disjointness of frequency support and rapid decay tell us that whenever $P,P'\in T$ and $\phi_{P}\in \Phi(P),\phi_{P'}\in \Phi(P'),$
\[ \ell_{I_P}=  \ell_{I_{P'}}\implies
\left|\langle \phi_{P},\phi_{P'}\rangle \right|  \lesssim |I_P|^{-1} \left\langle \frac{c_{I_P}-c_{I_{P'}}}{\ell_{I_P}}   \right\rangle^{-M},
 \qquad   \ell_{I_P}\neq  \ell_{I_{P'}}\implies \langle\phi_{P},\phi_{P'}\rangle = 0.
\]
This observation and standard kernel estimates tell us that the operator 
\[
f\mapsto H_T f\coloneqq  \sum_{P\in T} |I_P| \langle   f, \phi_{P} \rangle \varphi_P , \qquad  \phi_{P},\varphi_P\in \Phi(P)\quad \forall P \in T
\] 
and its adjoint are standard $L^2$-bounded Calder\'on-Zygmund operators. Thus, Calder\'on-Zygmund theory and the localization trick yield in particular that
\[
\frac{1}{|I_T|} \|H_T f\|_{1,\infty} \lesssim \llangle f \rrangle_{1,I_T}, \qquad 
\frac{1}{|I_T|^{\frac1p}} \|H_T f\|_{p} \lesssim_p \llangle f \rrangle_{p,I_T},
\]
the latter inequality being true for all $1<p<\infty$. In particular
\[
\size_{2}(W[f],T) \sim  {|I_T|^{-\frac12}} \|H_T f\|_{2}\lesssim\llangle f \rrangle_{2,I_T} \lesssim \|f\|_{\infty}
\] with 
$\phi_P, \varphi_P$   suitably chosen so that the first absolute equivalence holds.  We have just proved the outer estimate
\begin{equation}
\label{e:wpt0}
\|W[f ]\|_{L^{\infty}({J,\kappa,\size}_{2,\star})} \lesssim_p \|f\|_{\infty}.
\end{equation}
The more precise localized estimate of the next proposition may be proved using a semi-discrete analogue of $H_T$ and the John-Str\"omberg inequality. The argument is a variation on \cite[Prop.\ 9.3]{HytLac}.
Associate to a collection of tiles $ \mathbb{P}\subset \mathbb S $  and $ f \in L^{\infty}_0(\R)$ the quasinorms \begin{equation}
\label{e:localpnorm}
 [f]_{p,\mathbb{P}}\coloneqq\sup_{P \in \mathbb{P}} \inf_{I_P} \mathrm{M}_pf, \qquad 0<p<\infty .\end{equation}
\begin{proposition} \label{sizelemmacompressed}  $\displaystyle  
	\| W[f] \cic{1}_{\mathbb{P}} \|_{L^{\infty}({J,\kappa,\size}_{2,\star})} \lesssim \left\langle \mathrm{dist}(J, \supp f)\right\rangle^{-2^8} [f]_{1,\mathbb{P} }.$
\end{proposition}

\begin{proof}  There is no loss in generality with assuming $\mathbb P \subset \mathbb S^J$. For $\xi \in \R$, denote by $T_\xi=\{P \in \mathbb P: \xi\in  \omega_P^{\mathsf{p}(\kappa)} \}$. Note that $T_\xi$ is a tree with top data $(J,\xi)$. Then   
\begin{equation}
\label{e:proof341}  \| W[f] \cic{1}_{\mathbb{P}} \|_{L^{\infty}({J,\kappa,\size}_{2,\star})}\leq 2
\sup_{\xi \in \R} \sup_{T\subset T_\xi} |I_T|^{-\frac12} \left\|\langle f,\phi_P\rangle \cic{1}_T(P)\right\|_{\ell^2_P(\mathbb S^J)}
\end{equation}
for suitably chosen  $\phi_P\in \Phi(P)$.
So we fix $\xi$ and estimate $\sup_{T\subset T_\xi}   \left\|\langle f,\phi_P\rangle \cic{1}_T(P)\right\|_{\ell^2_P(\mathbb S^J)}$. By composing with modulations, we may reduce to $\xi=0$, and by \eqref{e:spectree} and finite splitting , we may also reduce to having $\#\{P\in T_\xi: I_P=I\}=1$ for all $I\in \mathcal I(T_\xi)$. Then 
\begin{equation}
\label{e:proof342}
\sup_{T\subset T_\xi} \frac{\left\|\langle f,\phi_P\rangle \cic{1}_T(P)\right\|_{ {\ell^2_P(\mathbb S^J)}}}{|I_T|^{\frac12} } \leq \sup_{\substack {K \in \mathcal D \\ K \subset J} }  \frac{ \left\|\langle f,\phi_P\rangle \cic{1}_{\{P\in T_\xi: I_P\subset K\}}\right\|_{\ell^2_P(\mathbb S^J)}}{|K|^{\frac12} }   = \left\| \sum_{I\in \mathcal I(T_\xi)} \langle f,\varphi_I\rangle h_{I} \right\|_{\mathrm{BMO} } 
\end{equation}
where we have set  $\varphi_I=\sqrt{|I_P|}\phi_{P}$ for the unique $P\in T_\xi$ with $I_P=I$, $h_I$ stands for the $L^2$-normalized Haar wavelet on $I$, and we mean the dyadic BMO. For $K\in \mathcal D, K \subset J$, let $\mathcal I^*(K)$ be the collection of maximal intervals in $I\in \mathcal I(T_\xi)$ with $I \subset K$. The John-Str\"omberg inequality, followed by disjointness of $I \in \mathcal I^*(K) $ tells us that
\begin{equation}
\label{e:proof343}
\begin{split}
\left\| \sum_{I\in \mathcal I(T_\xi)} \langle f,\varphi_I\rangle h_{I} \right\|_{\mathrm{BMO}} \lesssim \sup_{\substack {K \in \mathcal D \\ K \subset J} } \frac{1}{|K|}\left\| \sum_{\substack{I\in \mathcal I(T_\xi)\\I \subset K}}   \langle f,\varphi_I\rangle h_{I} \right\|_{1,\infty}  = \sup_{\substack {K \in \mathcal D \\ K \subset J} } \sum_{I\in \mathcal I^*(K)}\frac {\left\| H_{I,\mathrm{semi}} f  \right\|_{1,\infty}}{{|K|}}
\end{split}
\end{equation}
having set
\begin{equation}
\label{e:proof344}
H_{I,\mathrm{semi}} f\coloneqq \sum_{\substack{J\in \mathcal I(T_\xi)\\J \subset I}}   \langle f,\phi_J\rangle h_{J}.
\end{equation}
Standard kernel computations tell us that $H_{I,\mathrm{semi}}$ is also an $L^2$-bounded Calder\'on-Zygmund operator and in particular is uniformly of type weak-$(1,1)$. Combining with the localization trick on $I\in\mathcal I^*(K)  $,
\begin{equation}
\label{e:proof345}
\left\| H_{I,\mathrm{semi}} f  \right\|_{1,\infty} \lesssim |I| \llangle f\rrangle_{1,I} \lesssim |I| \inf_{ I} \mathrm{M}_1f \leq |I|  [f]_{1,\mathbb{P}}.
\end{equation} 
Inserting the estimate \eqref{e:proof345} into \eqref{e:proof343}, summing over the disjoint $I\in \mathcal I^*(K)$, and perusing \eqref{e:proof341}-\eqref{e:proof342} yields the partial bound
$\| W[f] \cic{1}_{\mathbb{P}} \|_{L^{\infty}({J,\kappa,\size}_{2,\star})} \lesssim  [f]_{1,\mathbb{P} }$. The additional decay factor may be easily obtained by a localization trick  followed by the partial result applied to  $f\chi_J$ in place of $f$.
\end{proof}
The following technical lemma will allow us to estimate the $ L^{\infty}({J,\kappa,\size}_{2,\star}) $ norm of the wave packet transform restricted to a collection $\mathbb P$ which is covered by a certain set of top data. It will not be used until Section \ref{s:lwpe}, but this is the most appropriate location for its proof. Notice that $T(I,\xi)$ appearing in the statement that follows is a $\kappa$-tree with top data $(I,\xi)$.
\begin{lemma} \label{treestructure} 
	Let $\mathbb P\subset \mathbb S$ and   $ \F\subset \mathcal D\times \R $ be a collection of top data covering $\mathbb P$, in the sense that
	\[
	\mathbb P=\bigcup_{(I,\xi)\in \F} T(I,\xi), \qquad T(I,\xi)\coloneqq\left\{P\in \mathbb P: I_P\subset I,\xi\in  \omega_P^{\mathsf{p}(\kappa) }\right\}.
	\]
	Then 
	\[ \left\| W[f] \cic{1}_{\mathbb P} \right\|_{L^{\infty}({J,\kappa,\size}_{2,\star})} \leq 2^{\frac\kappa 2}\ds{\sup_{(I,\xi) \in \F}} \; \size_{2,\star, {\kappa}}(W[f],T(I,\xi)).\] 
\end{lemma}
\begin{proof}There is no loss in generality with assuming $\mathbb P \subset \mathbb S^J$, and we do so.
Fix  a lacunary $ \kappa $-tree $  T \subseteq \mathbb{P} $ and let $(I_T,\xi_T)$ be its top data. Note that $(I_T,\xi_T)$ does not necessarily belong to $\F$.
Say that  $P \in \mathbb P ^{T,\star}$ if $P\in T$ and $I_P$ is a maximal element of   $\mathcal I(T)$ with respect to inclusion. By assumption, for each $P\in \mathbb P ^{T,\star}$ we may find $(I(P),\xi(P))\in \mathcal F$ with $I_P \subset I(P)$ and $\xi(P)\in \omega_P^{\mathsf{p(\kappa)}} $.  Clearly   
\[T=\bigcup_{P \in \mathbb P ^{T,\star}}T(P), \qquad T(P)\coloneqq\left\{Q=I_Q\times \omega_Q \in T: I_Q \subset   I_P \right\}\]
The fact that $T$ is a tree guarantees if $Q\in T(P)$ then   $\xi_T \in \omega_P^{\mathsf{p(\kappa)}}\cap \omega_Q^{\mathsf{p(\kappa)}} $, and comparing scales 
$\xi(P)\in  \omega_P^{\mathsf{p(\kappa)}}\subset \omega_Q^{\mathsf{p(\kappa)}}$. Therefore $T(P) $ is a $ \kappa $-lacunary tree with top data $ \left(I_P,\xi(P)\right) $, whence the inclusion $T(P)\subset T(I(P),\xi(P))$ for all $P\in \mathbb P ^{T,\star}$, and 
\[
\size_{2}(W[f],T(P)) \leq \size_{2,\star,\kappa}(W[f],T(I(P),\xi(P)))\leq \sup_{(I,\xi)\in \mathcal F} \size_{2,\star,\kappa}(W[f],T(I,\xi)).
\]
Using \eqref{e:spectree} and disjointness of the maximal elements of $\mathcal I(T)$, which are all contained in $I_T$, \[ 
\begin{split}
&	\size_{2}(W[f],T)\leq  \left( \frac{1}{|I_T|} \sum_{P \in {\mathbb P}^{T,\star}} |I_P|\big[
\size_{2}(W[f],T(P))\big]^2 \right)^{\frac12} 
  \leq  2^{\frac \kappa 2}\sup_{(I,\xi)\in \mathcal F} \size_{2,\star,\kappa}(W[f],T(I,\xi))
\end{split} 
\] 
which completes the proof of our main claim.\end{proof}

\subsection{Local $L^2$-bound for maximal modulations  via wave packet estimates}
  In this paragraph, as a motivating example, two more outer $L^p$ estimates for the wave packet transforms \eqref{e:wpt}-\eqref{e:wpt} are stated and combined into a proof of  $L^p$-boundedness for the maximal modulated singular multiplier of \eqref{carleson} in the local $L^2$-range.
The first concerns the wave packet transform \eqref{e:wpt} \begin{proposition} \label{usual}Let $J\in \mathcal D$ and   $f\in L^{\infty}_{0}(\R)$.  Then	
	\begin{align}
\label{usual2}	&\|W[f]\|_{L^{2,\infty}({J,\kappa,\size}_{2,\star})} \lesssim   \llangle f \rrangle_{2,3J} \\ \label{usualp} & \|W[f]\|_{L^{p}({J,\kappa,\size}_{2,\star})} \lesssim_p \llangle f \rrangle_{p,3J}, \qquad 2<p\leq\infty.
	\end{align}
\end{proposition}
The bound \eqref{usual2}  is  a restatement of \cite[Theorem 5.1]{DT2015}, see also \cite{CDPOUBHT,DPOu}.   Once \eqref{usual2} is at disposal, \eqref{usualp} follows immediately from its outer $L^p$ interpolation with e.g.\ \eqref{e:wpt0}; an appropriate interpolation theorem is \cite[Prop.\ 3.5]{DT2015}.   A similar, but broader set of estimates is available for the $  L^{p}(\size_{\mathsf{C}},J) $ norms of \eqref{e:wptA}. As anticipated, the outer norms below refer to the case $\kappa=1$.
\begin{proposition} \label{usualforA} Let $J\in \mathcal D$ and   $f\in L^{\infty}_{0}(\R)$.  Then	\[
	\begin{split}
	&\|A[f\cic{1}_{3J}]\|_{L^{1,\infty}({J,1,\size}_{\mathsf{C}})} \lesssim \langle f \rangle_{1,3J}, \\ & \|A[f\cic{1}_{3J}]\|_{L^{p}({J,1,\size}_{\mathsf{C}})} \lesssim_p \langle f \rangle_{p,3J}, \qquad 1<p\leq\infty.
	\end{split}
	\]
\end{proposition} Proposition \ref{usualforA} is obtained as a consequence of the localized estimate \eqref{e:modifemb1} of Proposition \ref{p:modifemb}. We send to  Section \ref{s:modifemb} for statements and proofs.
Propositions \ref{usual} and \ref{usualforA} may be combined to prove the estimate
\begin{equation}
\label{e:introest}\mathsf C_{\mathbb P}(f_1,f_2)\lesssim_{p}  \|f_1 \|_{p} \|f_2 \|_{p'}, \qquad 2<p<\infty
\end{equation} 
uniformly over all $f_1,f_2\in L^\infty_0(\R)$ and finite   $\mathbb P\subset \mathbb S $.
 In turn, via \eqref{e:modelred}, \eqref{e:introest} entails the $L^p(\R)$-boundedness of \eqref{carleson} in the same range.
 \begin{proof}[Proof of \eqref{e:introest}] Fix $f_1,f_2\in L^\infty_0(\R)$ and a finite $\mathbb P$. Using grid property (ii),   find $  J \in \D$ such that, denoting  $J_j= J+j|J|
$, and setting $\mathbb{P}_j\coloneqq \mathbb{P } \cap \mathbb{S}^{J_j}$, 
there holds \[\mathbb P=\mathbb{P}_{-1} \cup \mathbb{P}_{0} \cup \mathbb{P}_1, \qquad \supp f_1,\supp f_2 \subset 3J_j \quad \forall j=0,\pm1.\] 
    The easy consideration $\size_\infty(F,T)\leq \size_{2,\star,1}(F,T)$ and the definitions tell us that 
  \begin{equation}\label{e:carlholder}
 \begin{split}
  \size_{1} (F_1F_2,T)
  &  \leq\size_2(F_1,T^{\mathsf{lac}})\size_2(F_2,T^{\mathsf{lac}}) +\size_\infty(F_1,T^{\mathsf{ov}})\size_1(F_2,T^{\mathsf{ov}})
	\\ &\leq 2   \size_{2,\star,1}(F_1,T) \size_{\mathsf{C}}(F_2,T).
\end{split}
 \end{equation}  so that a form of \eqref{e:submulti} is verified.
Applying the outer H\"older inequality to $F_1=W[f_1], F_2=A[f_2]$ in the form of \eqref{usualholder} followed by Propositions \ref{usual} and \ref{usualforA} thus leads to \[ \begin{split}
			&\quad \mathsf{C}_{\mathbb{P}_i}(f_1,f_2) \leq
		\left\| F_1F_2\right\|_{\ell^1(\mathbb S^{J_i} )} 
	 \lesssim  |J_i| \left\| W[f_1  ] \right\|_{L^p(J_i,1,\size_{2,\star})} \left\| A[f_2 ] \right\|_{L^{p'}(J_i,1,\size_{\mathsf{C}})} 	\\ &  \lesssim |J_i| \l f_1 \r_{p,3J_i} \l f_2 \r_{p',3J_i}\lesssim  \left\| f_1 \right\|_p \left\| f_2 \right\|_{p'}
	 	\end{split}\] and the proof is completed by the observations that $\mathsf{C}_{\mathbb{P}} = \mathsf{C}_{\mathbb{P}_{-1}}+ \mathsf{C}_{\mathbb{P}_0}+\mathsf{C}_{\mathbb{P}_1}$.
 \end{proof}

\section{Localized embeddings for the modified  wave packet transforms} \label{s:modifemb}
This section contains the statement and proof of the embedding theorems for the modified wave packet transform \eqref{e:wptA}, see Proposition \ref{p:modifemb}. The analysis behind this proposition is  essentially based on a combination of the \emph{tree} and \emph{mass} lemmata from \cite{LT}. We claim no particular originality, but  choose to present a full argument given the additional complications brought by the explicit dependence on $N(\cdot)$ of the wavelets in the  map \eqref{e:wptA}, cf.\ also the definition of the wavelet classes $\Psi(P)$ from \eqref{e:refwpsO}. To handle this dependence, we borrow a continuity estimate idea from the paper \cite{LaceyLi} on Stein's conjecture for the Hilbert transform along vector fields.

\begin{remark} \label{r:formal}
Before we begin, we make the standing assumption that the function $f$  playing the role of the argument in \eqref{e:wptA} belongs to $L^{\infty}_0(\R)$ and that     $ \mathbb{P}$ is a finite subset of the collection of all tiles $ \mathbb{S} =\mathbb{S}_{\mathcal D,\mathcal D'}$. The finiteness assumption in the estimates does not change the scope of our applications, and may in fact be removed via a limiting argument when additional regularity assumptions on $f$ are posed; for instance $f \in \mathcal C^2_0(\R)  $ will suffice.
\end{remark}
 \begin{proposition} \label{p:modifemb} We have
 \begin{equation}
\label{e:modifemb1}
 \left\| A[f] \cic{1}_{\mathbb{P}} \right\|_{L^{p,\infty}\left({J,1,\size}_{\mathsf{C}}\right)} \lesssim [f]_{1,\mathbb{P}} , \qquad 1\leq p\leq \infty
    \end{equation} with uniform constant. In particular the above estimate  yields the control  \begin{equation}
\label{e:modifemb2}\left\| A[f] \cic{1}_{\mathbb{P}} \right\|_{Y^{p,\infty}\left({J,1,\size}_{\mathsf{C}}\right)} \lesssim [f]_{1,\mathbb{P}}, \qquad 1\leq p\leq \infty.
\end{equation}
\end{proposition} In Proposition \ref{p:modifemb}, as anticipated in Section \ref{s:outer}, the tree parameter   $\kappa$ equals 1 and all trees referred to below are $1$-trees, without further explicit mention.
The  proposition is proved  by combining the next two lemmata, involving the auxiliary quantity
\begin{equation}
\label{e:dense}
\dense(f,\mathbb{P})\coloneqq \sup_{P \in \mathbb{P}} \sup_{\substack{P \lesssim  P' \\ P' \in \mathbb{S}}}  \llangle  f \cic{1}_{N^{-1}( \omega_{P'}^{\mathsf{p}(1)})  }\rrangle_{1,I_{P'}}   
\end{equation}
defined  e.g. for $f\in L^\infty_0(\R)$ and $\mathbb P\subset \mathbb S$.
 The order relation in \eqref{e:dense} is a modification of the Fefferman ordering defined by
\begin{equation}
\label{e:feffor}
P\lesssim_\kappa  P' \iff I_P \subset I_{P'}, \;  \omega_{P'}^{\mathsf{p}(\kappa)} \subset   \omega_{P}^{\mathsf{p}(\kappa)}.
\end{equation}
As we use \eqref{e:feffor} with $\kappa=1$ throughout this section, we write $\lesssim$ instead of $\lesssim_1 $.
\begin{remark} \label{r:densecont} A moment's thought yields $\dense(f,\mathbb{P})\lesssim [f]_{1,\mathbb P}$ uniformly   over $\mathbb P\subset \mathbb S$. 
\end{remark}
\begin{lemma}
	\label{densitydom}$\displaystyle \left\| A[f] \cic{1}_{\mathbb{P}} \right\|_{L^{\infty}({J,1,\size}_{\mathsf{C}})} \lesssim \dense(f,\mathbb{P}).$
\end{lemma}
 \begin{lemma} \label{l:densitycarl} Let $\mathbb P \subseteq \mathbb{S}$    and $\delta>0$. There exists a decomposition
		$
		\mathbb P =  \mathbb P_- \cup \bigcup_{T\in \mathcal F} T,
	$
		where $\dense(f,\mathbb P_-)\leq \delta$,  each $T$ is a tree with top interval $I_T$, and the forest $\mathcal F=\mathcal F(\delta,f)$ satisfies  		\begin{equation}
			\label{e:topsstop}\frac{ \delta }{|J|}\sum_{\substack{T\in \F \\ I_T \subset J}} |I_T |\lesssim   \inf_J \mathrm{M}_1 f, \qquad \forall J \in \mathcal D.
		\end{equation}
	\end{lemma}
The proofs of Lemmata \ref{densitydom} and \ref{l:densitycarl}   are respectively postponed to Subsections \ref{pf:densitydom} and \ref{pf:densitycarl}. We now show how a combination of these yields Proposition \ref{p:modifemb}. Fix $J \in\mathcal D$, $\mathbb P\subset \mathbb S$ .
The bound \eqref{e:modifemb1} is an immediate consequence of 
\begin{equation}
\label{e:modifemb3}
\sup_{t>0}\max\{1,t\} \left(A[f] \cic{1}_{\mathbb{P}}\right)^{*,\size_{\mathsf{C}}}(t) \leq  C[f]_{1,\mathbb{P}}\end{equation}
where $C$ is an absolute constant explicitly computed below and $\mathbb P\subset \mathbb S^J$. The range $t\leq 1$ of  estimate \eqref{e:modifemb3} is readily obtained by combining Remark \ref{r:densecont} with the conclusion of Lemma \ref{densitydom} and choosing $C$ to be larger than the product of the respective absolute implicit constants. Now, notice that the right hand side of  \eqref{e:topsstop} is also controlled by $[f]_{1,\mathbb P}$. 
Applying Lemma \ref{l:densitycarl}  to $  \mathbb{P} $ with the choice 
  $ {\delta=C \frac{[f]_{1,\mathbb{P}}}{t}} $, provided $C$ is larger than twice the implicit constant in \eqref{e:topsstop} yields  \eqref{e:modifemb3} in the range $t\geq 1.$

\subsection{Proof of Lemma \ref{densitydom}} \label{pf:densitydom} The proof of the Lemma consists in showing that   \begin{equation}
	\label{e:pfLemma431}
	\size_{\mathsf{C}}(A[f] ,T) = \frac{1}{|I_T|} \sum_{P\in T^{\mathsf{ov}}} |I_P| A[f](P) + \left(\frac{1}{|I_T|} \sum_{P\in T^{\mathsf{lac}}} |I_P| A[f](P)^2 \right)^{\frac12}  \lesssim \dense(f,\mathbb{P}) 
\end{equation}
whenever $ T \in \mathcal T^{J,\star}$  is a  tree with $T \subset \mathbb{P} $. For any such tree,  we introduce the support intervals
\[
\Omega^{\mathsf{b}}(T)\coloneqq\{ \omega_P^{\mathsf{b}}: P \in  T  \}.\]
  We may assume, by splitting, that $T$ is a type 2 tree, which means that $\omega_P$ is the right child of its dyadic parent $\omega_P^{\mathsf{p}(1)}$ for all $P\in T$.
 Lemma \ref{lacpartislac} thus tells us   that the collection $\Omega^{\mathsf{b}}(T^{\mathsf{ov}})$ consists of   pairwise disjoint intervals, while $T^{\mathsf{lac}}$ is a lacunary tree, so that in particular $\Omega^{\mathsf{b}}(T^{\mathsf{lac}})$ is a nested collection of intervals containing $\xi_T$. This follows immediately by combining 
\[
\xi_T\in \omega_P^{\mathsf{p}(1)} \; \forall P \in T, \qquad   \omega_P\neq \omega_{P'}\implies \omega_P\cap \omega_{P'}= \varnothing\quad  \forall P,P'\in  T^{\mathsf{lac}} .\]
Accordingly, the quantity $\delta(x) \coloneqq \inf \big \{\ell_{\omega_P}: N(x)\in \omega_P^{\mathsf {b}}, P\in T^{\mathsf{lac}} \big\}$
records the minimal active frequency scale of $T^{\mathsf{lac}}$ at each $N(x)\in \R$ and satisfies
\begin{equation}
\label{e:minsc}
|N(x)-\xi_T| \leq \delta(x), \qquad x\in \R.
\end{equation}

  In estimating both contributions, a key role is played by the collection $ \mathcal{G} $ of maximal elements  in $ \{  G\in \D : 3G \not \supseteq I_P \,\forall P \in T   \} $.
 Accordingly, for  $G\in \mathcal G$, $\mathsf{j}\in \{\mathsf{ov}, \mathsf{lac}\}$ decompose
\[
T^{\mathsf{j}}=T^{\mathsf{j},+}_G \cup T^{\mathsf{j},-}_G , \qquad T^{\mathsf{j},+}_G=  \{P\in T^{\mathsf{j}}:  \scl(P) > \ell_G\}, \qquad T^{\mathsf{j},-}_G= \{P\in T^{\mathsf{j}}:  \scl(P) \leq  \ell_G\}.
\]
We begin to estimate the $T^{\mathsf{ov}}$ term in \eqref{e:pfLemma431}. Using the  definition and the fact that $\mathcal {G}$ is a partition of $\R$ leads to
\begin{equation}
	\label{e:basicpf1}
	\begin{split}
		&   \sum_{P \in T^{\mathsf{ov}}}|I_P| A[f](P)\leq  2\sum_{*\in \{+,-\}} \sum_{G \in \mathcal{G}}   \sum_{ P \in T^{\mathsf{ov,*}}_G  }  |I_P|  \langle f, \varphi_P
		\cic{1}_{G\cap N^{-1}( \omega_{P}^{\mathsf{b}} ) }  \rangle, \qquad \varphi_P\coloneqq \psi_P(\cdot,N(\cdot) )    
\end{split}   \end{equation}
for suitable  $ \psi_P \in \Psi(P)   $. Note that $\varphi_P$ are not standard wavelets as they carry the dependence on the measurable function $N$ from the second argument of $\psi_P$.  The basic estimate
\begin{equation}
	\label{e:basicpf2}
	|\langle f, \varphi_P
		\cic{1}_{G\cap N^{-1}( \omega_{P}^{\mathsf{b}} ) }  \rangle| \lesssim \chi_{I_P}^{10}(c_G)\dense(f,\mathbb P)
\end{equation}
reveals that the $*=-$ sum in \eqref{e:basicpf1} is a tail term. Indeed, also relying on the defining property of $\mathcal {G}$ for the first estimate, and later on \eqref{e:spectree},  
\begin{equation}
	\label{e:taildense}
	\begin{split} 
		&\quad \sum_{G\in \mathcal G}
		\sum_{  P \in T^{\mathsf{ov},-}  } |I_P|\langle f, \varphi_P
		\cic{1}_{G\cap N^{-1}( \omega_{P}^{\mathsf{b}} ) }  \rangle 
		\lesssim  
		\dense(f,\mathbb P)\sum_{G\in \mathcal G} \sum_{k\geq 0} \sum_{\substack{ P \in T^{\mathsf{ov}} \\ |I_P| = 2^{-k}|G|, I_P \cap G=\varnothing  }} |I_P| \chi_{I_P}^{10}(c_G)
		\\ & \lesssim \dense(f,\mathbb P)\sum_{G\in \mathcal G}  |G| \chi_{I_T}^{9}(c_G) \lesssim  \dense(f,\mathbb P)\int  \chi_{I_T}^{9} \lesssim \dense(f,\mathbb P) |I_T|
		\end{split}\end{equation}
which is compliant with \eqref{e:pfLemma431}.
 The $*=+$ term is estimated as follows. First, note that $T^{\mathsf{ov},+}(G)=\varnothing$ unless $G\subset 9I_T$ and there exists $P(G)\in T $ with \[
\scl(P(G))=2\ell_G, \qquad \dist(G,I_{P(G)})\leq \scl(P(G)) .\] Let $P'(G)\in \mathbb S_{\mathcal D,\mathcal D'}$ be the unique tile with $I_{P'(G)}=I_{P(G)}$ and $\xi_T\in \omega_{P'(G)}$. As  the intervals $\Omega^{\mathsf{b}} (T^{\mathsf{ov},+}(G))$ are pairwise disjoint and contained in $  {\omega_{P'(G)}^{\mathsf{p}({1})}}$,
\begin{equation}\label{ovlessthantime}\begin{split}
		&\quad \sum_{G \in \mathcal{G}} \sum_{ P \in T^{\mathsf{ov},+}_G }  |I_P| \langle f, \varphi_P
		\cic{1}_{G\cap N^{-1}( \omega_{P}^{\mathsf{b}} ) }  \rangle \lesssim \sum_{\substack{G \in \mathcal{G}\\  G \subset 9I_T}} |G|   \llangle  f \cic{1}_{N^{-1}( {\omega_{P'(G)}^{\mathsf{p}(1)}})  }\rrangle_{1,I_{P'(G)}} \\ &\lesssim \sum_{\substack{G \in \mathcal{G}\\ G \subset 9I_T}} |G| \dense(f,\{P(G)\})   \lesssim  \dense(f,\mathbb P) |I_T|,
	\end{split}
\end{equation}
which also complies with \eqref{e:pfLemma431}. The $\mathsf{ov}$ term in \eqref{e:basicpf1} is thus fully handled.

We move onto the $\mathsf{lac}$ term in \eqref{e:basicpf1}.  
 With the same notation of the  $ T^{\mathsf{ov}} $ the term, we estimate \begin{equation}\label{eqlac}  \begin{split}
		& \sum_{P \in T^{\mathsf{lac}}} A[f](P)^2 |I_P| \leq 2\sum_{*\in \{+,-\}} \sum_{G \in \mathcal{G}}   \sum_{ P \in T^{\mathsf{ov,*}}_G  }     A[f](P)   \langle f, \varphi_P
		\cic{1}_{G\cap N^{-1}( \omega_{P}^{\mathsf{b}} ) }  \rangle |I_P|.
	\end{split}
\end{equation}
The $  *=- $ sum in \eqref{eqlac} is handled  along the lines of \eqref{e:taildense}, with an additional application of \eqref{e:basicpf2}: we omit the details. The rest of the analysis deals with the $*=+$  summand in \eqref{eqlac}.
The explicit dependence of $\varphi_P=\psi_P(\cdot, N(\cdot))$ on $N(\cdot)$ prohibits us to use orthogonality methods directly. This is obviated by replacing  $\varphi_P$ with the standard wavelets 
\[
\phi_P\coloneqq \psi_P(\cdot, \xi_T)\in \Phi(P), \qquad P \in T^{\mathsf{lac}}.
\] 
Setting $\zeta_P\coloneqq|I_P|[\varphi_P -\phi_P]=|I_P|[\psi_P(\cdot, N(\cdot))-\phi_P],   P \in T^{\mathsf{lac}}$,
th error term created  by the replacement is  
\begin{equation}
\label{e:error}\begin{split} & \quad
\sum_{G \in \mathcal{G}, G \subset 9I_T} \sum_{\substack{P \in T_G^{\mathsf{lac,+}}}}  A[f](P) \left|\langle f,
		\zeta_P\cic{1}_{G\cap N^{-1}( \omega_{P}^{\mathsf{b}} ) } \rangle \right|  
		\\ & \lesssim \dense(f,\mathbb{P}) \sum_{G \in \mathcal{G},G \subset 9I_T} \left\langle |f| \cic{1}_{G\cap N^{-1}(\omega_{P'(G)}^{\mathsf{p}({1})})}, \sum_{\substack{P \in T_G^{\mathsf{lac,+}}}} |  \zeta_P |\cic{1}_{\omega_{P}^{\mathsf{b}}}(N(\cdot)) \right\rangle 
		\\ &\lesssim \sum_{G \in \mathcal{G}, G \subset 9I_T} |G| \dense(f,\mathbb{P})^2 \lesssim |I_T| \dense(f,\mathbb{P})^2
\end{split}
\end{equation}
For the passage to the second line, note that the intervals $\Omega^{\mathsf{b}} (T^{\mathsf{ov},+}(G))$ are   all contained in $  {\omega_{P'(G)}^{\mathsf{p}({1})}}$. The subsequent step was obtained via a Lipschitz estimate in the second argument of $  \psi_P \in \Psi(P) $  and subsequently taking advantage of \eqref{e:minsc}, so that
\[
\sum_{\substack{P \in T^{\mathsf{lac}}}} |  \zeta_P(x)| \cic{1}_{\omega_{P}^{\mathsf{b}}}(N(x))  
\lesssim  \sum_{ P \in T^{\mathsf{lac}} : \ell_{\omega_P}\geq \delta(x) } \chi_{I_P}(x) \frac{\delta(x)}{\ell_{\omega_P}}  \lesssim 1.
\]
We are left with estimating the $*=+$  summand in \eqref{eqlac}, where   the  $\varphi_P$ have been replaced by the almost orthogonal wavelets $\phi_P$. A principal role   is played by the tree operator
\[
H_T f \coloneqq \sum_{P\in T^{\mathsf{lac}}}   |I_P|A[f](P) \phi_P.
\] 
 As $\xi_T\in \omega^{\mathsf{b}}$ for all $\omega\in\Omega(T^{\mathsf{lac}})$, the intervals  $\Omega(T^{\mathsf{lac}})$ form a lacunary sequence, that is \[
  \omega\subset\left\{\xi\in \R:\textstyle\frac{\ell_{\omega}}{2}\leq \dist(\xi,\xi_T)\leq 2 \ell_{\omega}\right\}\quad \forall \omega\in \Omega(T^{\mathsf{lac}}), \;\ell_\omega > \ell_\star\coloneqq \min\left\{\ell_{\omega'}:\omega'\in\Omega(T^{\mathsf{lac}})\right\}.\]
  For $\alpha \in \{\ell_\omega:\omega\in \omega(T^{\mathsf{lac}})\}$, let 
$\Psi_{\alpha}$ be   even, real valued Schwartz functions with \[1_{[\frac12\alpha,2\alpha]}\leq \widehat \Psi_{\alpha}\leq 1_{[\frac25\alpha,\frac{11}{5}\alpha]}  \quad \alpha >\ell_\star ,\qquad   1_{[0,2\ell_\star]}\leq \widehat \Psi_{\ell_\star} \leq 1_{[0,\frac{11}{5}\ell_\star]}  .\]
Assuming  that $\{\ell_\omega:\omega\in T^{\mathsf{lac}}\}$ are separated by a factor of 4, and arguing by finite splitting otherwise, we obtain for all $\ell_\star\leq \alpha\leq \beta$,
\[\begin{split} &
\bigg|H_{T,\alpha,\beta} f\coloneqq
\sum_{\substack{P\in T^{\mathsf{lac}}\\ \alpha \leq \ell_{\omega_P} < \beta  }}  |I_P|A[f](P) \phi_P = \left[\Psi_\beta - \Psi_\alpha\right]* H_{T} f\bigg| \lesssim \mathrm{M} [H_{T} f] ,  
\end{split}\] 
 due to the frequency support conditions $\phi_P \subset \omega_P$. 
  Relying on the definition of $\delta(\cdot)$, cf.\ \eqref{e:minsc},  the modified $*=+$  summand in \eqref{eqlac} is then estimated by
\[
\begin{split} &\quad 
 \sum_{\substack{G \in \mathcal{G} \\ G \subset 9I_T}}   \int_G  |f| \cic{1}_{ N^{-1}(\omega_{P'(G)}^{\mathsf{p}({1})})} \left|H_{T,\delta(\cdot),\frac{1}{\ell_G}}   f\right| \lesssim \dense(f,\mathbb{P})^2    \int_{9I_T} \mathrm{M} [H_{T} f] \\ & \lesssim |I_T|\dense(f,\mathbb{P})^2    \|H_{T} f\|_2  \lesssim 
  |I_T|\dense(f,\mathbb{P})^2\left( \sum_{P\in T^{\mathsf{lac}}} |I_P| A[f](P)^2 \right)^{\frac12}. 
\end{split}
\]  
Balancing out the obtained bounds  completes the estimation of  $\mathsf{lac}$ term in \eqref{e:basicpf1}, and in turn, the proof of Lemma \ref{pf:densitydom}.

\subsection{Proof of Lemma \ref{l:densitycarl} } \label{pf:densitycarl}	
	The selection of the trees $T\in \mathcal F$ and consequent estimation of $\dense(f,\mathbb P_-)$   is identical to \cite[Proposition 3.1]{LT} and is thus omitted. To prove \eqref{e:topsstop}, it suffices to show that whenever $ \mathbb P'\subset \mathbb S^J$ is a set of pairwise incomparable tiles with respect to \eqref{e:feffor}    	\begin{equation} \label{e:topsdensek} \inf_
	{P\in \mathbb P'}\big\llangle  f \cic{1}_{N^{-1}( \omega_{P}^{\mathsf{p}(1)})  }\big\rrangle>\delta \implies 
		\sum_{P\in \mathbb P'} |I_P| \lesssim \delta^{-1} |J| \inf_J \mathrm{M}f.
	\end{equation}
	 Due to the premise of \eqref{e:topsdensek}, for each $P \in \mathbb P'$ there exists $k=k_P\geq 0$ with the property that
	\[
	\int_{2^kI_P \cap N^{-1}(\omega_{P}^{\mathsf{p}(1)})} |f| \geq 2^{6k} \delta|I_P|.
	\] and $k_P$ is minimal with this property.
	Let $\mathbb P'_k$ be the collection of all $P\in \mathbb P'$ with $k_P=k$. Perform the following iterative selection.
	Initialize $\mathbb A\coloneqq \mathbb P'_k, \mathbb B=\varnothing$.  Among those  $P^\star\in \mathbb A $ with \[
	2^kI_{P^\star} \times\omega_{P^{\star}}^{\mathsf{p}(1)}\cap 2^kI_{P} \times \omega_{P}^{\mathsf{p}(1)}=\varnothing \qquad \forall P \in \mathbb B\]
select  one with $\scl(P^\star)$ maximal, and 
	 set  $\mathbb A\coloneqq \mathbb A\setminus \{P^\star\}$, $\mathbb B\coloneqq \mathbb B\cup \{P^\star\}$. Repeat until no such $P^\star\in \mathbb A$ is available.
	At this point, we may partition $\mathbb P'_k=\bigcup\{\mathbb P'_k(P^\star):P^\star\in \mathbb B\}$ where $P\in \mathbb P'_k(P^\star)$ if
	\[ 2^kI_{P} \times \omega_{P}^{\mathsf{p}(1)}\cap 
	2^kI_{P^\star} \times\omega_{P^{\star}}^{\mathsf{p}(1)}\neq \varnothing , \qquad \scl(P) \leq \scl(P^\star).
	\]
	Notice that if $P,P'\in \mathbb P'_k(P^\star) $ then $\omega_{P}^{\mathsf{p}(1)}\cap  \omega_{P'}^{\mathsf{p}(1)} \supset \omega_{P^{\star}}^{\mathsf{p}(1)} $, and $P,P'$ are incomparable, so that the intervals $\{I_P: P \in \mathbb P'_k(P^\star)\}$ are pairwise disjoint and contained in $2^{k+2} I_{P^\star}$. We then have
	\[
	\begin{split} &\quad
		\sum_{P\in \mathbb P'_k} |I_P| \leq \sum_{P^\star\in \mathbb B} \sum_{P\in \mathbb P'_k(P^\star)} |I_P| \lesssim 2^k    \sum_{P^\star\in \mathbb B} |I_{P^\star}| \lesssim 2^{-5k}\delta^{-1}\sum_{P^\star\in \mathbb B}\int_{2^kI_{P^\star} \cap N^{-1}(\omega_{P^{\star}}^{\mathsf{p}(1)})} |f|
		\\
		& \lesssim 2^{-4k} \delta^{-1} |J| \langle f\rangle_{1,2^{k+2} J} \lesssim 2^{-4k}\delta^{-1} |J| 
		\inf_J \mathrm{M}f.
	\end{split}
	\]
	To pass to the second line, note that the sets $2^kI_{P^\star} \cap N^{-1}(\omega_{P^{\star}}^{\mathsf{p}(1)})$, $P^\star\in \mathbb B$ are pairwise disjoint  and contained in $2^{k+3} J$. Then \eqref{e:topsdensek} follows by summing over $k$.
 
\section{Localized wave packet estimates near $L^1$} \label{s:lwpe}
If  the local $L^2$-averages of $f$  are under control, we may combine the bound of Proposition \ref{sizelemmacompressed} with  \eqref{usual2} in the single localized estimate
\begin{equation}
\label{e:2emb} 
	\left\|W[f]\cic{1}_{\mathbb P}\right\|_{Y^{2,\infty}({J,\kappa,\size}_{2,\star})} \leq C_{\kappa}  [f]_{2,\mathbb P}.
\end{equation}
The quantities $[f]_{p,\mathbb P}$ have been introduced in \eqref{e:localpnorm}.
This section contains the statement and main line of proof of a localized estimate for the wave packet transform in terms of  local $L^p$ norms in the range $1< p\leq 2$, with good control on the estimate as $p\to 1^+$. Throughout the remainder of this section, we  enforce the formal assumptions of  Remark \ref{r:formal} without further explicit mention.

The main result of \cite{DPOu}, a first substitute for  \eqref{e:2emb} outside local $L^2$, is recalled in the next proposition.
\begin{proposition} \label{p:embdpou}Let $1<p\leq 2$.  For all $t>1$ there exists $C_{t,p,\kappa}>1$ such that the following holds. Let $J$ be any interval, $f\in L^{\infty}_{0}(\R)$ and $\mathbb P\subset \mathbb S$.
	Then
	\[
	\left\|W[f\cic{1}_{3J}]\cic{1}_{\mathbb P}\right\|_{L^{tp'}({J,\kappa,\size}_{2,\star})} \leq C_{t,p,\kappa}  [f]_{p,\mathbb P}.	\]
\end{proposition}  Proposition \ref{p:embdpou} has been used to prove sparse  and localized estimates for the Carleson operator \cite{DPDoU} and the bilinear Hilbert transform \cite{CDPOUBHT}. However, an inspection of the proof shows that having fixed $t>1$, the constant $C_{t,p}$ blows up polynomially  in $(p-1)^{-1}$ as $p\to 1^+$. 

 The next theorem, which is the main technical novelty of this work, provides us with a substitute embedding that does not blow up near $p=1$. Remark \ref{weaklpineq} tells us  that the norms $X^{tp',\infty}_2$ are weaker than  the ones appearing on the left hand side of Proposition \ref{p:embdpou}. Nonetheless, the generalized H\"older inequality of Proposition \ref{p:eHolder} makes Theorem \ref{mainemb} applicable for our purposes.
 
 \begin{theorem} \label{mainemb}   For all $t>1$ there exists $C_{t,\kappa}>1$ such that the following holds. Let $1<p\leq 2,$ $J$ be any interval, $f\in L^{\infty}_{0}(\R)$ and $\mathbb P\subset \mathbb S$ . Then 
\begin{equation}
\label{e:mainemb1}
	\left\|W[f]\cic{1}_{\mathbb P}\right\|_{X^{tp',\infty}_2({J,\kappa,\size}_{2,\star})} \leq  C_{t,\kappa}  [f]_{p,\mathbb P}. 
\end{equation}
\end{theorem}
\begin{remark} \label{smoothnesstprime} We clarify a delicate point in the statement of  Theorem \ref{mainemb}. Fixing $t$, the smoothness level of the wave packet transform, as defined in \eqref{e:wpt}, required for   Theorem \ref{mainemb}  must be greater or equal to, say, $ M  = 10 \cdot \lceil 2^{8}t' \rceil. $  Theorem \ref{mainemb} will be  applied below with the fixed choice   $ t=2 $, so that   a fixed level of smoothness, say $  M =10 \cdot2^{9}, $ is sufficient.

\end{remark}
\begin{remark}[On sharpness of Theorem \ref{mainemb}]Theorem \ref{mainemb} is sharp in the sense that it captures the necessary linear growth in $p'$ of the target exponent in the left hand side. More precisely,  if the function $q=q(p)$ is such that the estimate\[
\left\|W[f ]\cic{1}_{\mathbb P}\right\|_{X^{q(p),\infty}_2({J,\kappa,\size}_{2,\star})} \leq  C  [f]_{p,\mathbb P}
\]
holds with a uniform constant $C$ for all $1<p\leq 2$, then 
$
\liminf_{p\to 1^+} \frac{q(p)}{p'} =t 
$ 
 for some $t\geq 1$. This is easily seen by  testing the uniform inequality on the family of functions
$ f=
 f_n\coloneqq  \sum_{j=0}^n \Mod_{4j} \phi
$
 where $\phi$ is any smooth function with $\widehat{\phi}(0)=1$ and $\supp \phi\subset [0,1]$, and on the collections $\mathbb P=\mathbb P_n=\{P: I_P=[0,1), \omega_P=[j,j+1), j=1,\ldots, 4n \}$. 
 
If the wave packet transform is replaced with its Walsh group analogue, the methods of e.g.\ \cite{DPCRM} can be used to obtain the case $t=1$ of \eqref{e:mainemb1} as well. While this does not seem to be achievable with the techniques of the present paper for the Fourier case,  we stress that  being able to take $t=1$ in \eqref{e:mainemb1} would not lead to any improvements in our quantified estimates for the Carleson operator.
\end{remark}

 The proof of Theorem \ref{mainemb} occupies the remainder of this section and is structured as follows.   Subsection \ref{ss:rwc} introduces a  generalization of the wavelet classes  $\Phi^M(P)$ of \eqref{e:refwps} where the compact frequency support assumption is relaxed to requiring instead  vanishing moments with respect to a fixed frequency.  
 
  \subsection{Relaxed wavelet classes}  \label{ss:rwc} For an interval $I\subset \R$   and $\xi\in \R$, define the normalized classes
\begin{equation}
\label{e:refwpsrel}
\begin{split}
&\Theta_1^{M} (I,\xi) \coloneqq\left\{\mathrm{Mod}_{\xi} \mathrm{Tr}_{c(I)}
\mathrm{Dil}_{\scl(I)}^1 \vartheta :\vartheta\in \Phi^M\right\}
\\
&  \Theta_0^{M} (I,\xi) \coloneqq\left\{\mathrm{Mod}_{\xi} \mathrm{Tr}_{c(I)}
\mathrm{Dil}_{\scl(I)}^1 \vartheta :\upsilon\in \Phi^M, \widehat \upsilon(0)=0\right\}.
\end{split}
\end{equation} As usual we drop the $M$ when irrelevant or clear from context.
If $  P \in  \mathbb{S} $ is a tile and $\phi_P \in \Phi (P) $ we have the inclusions
\[
\xi \in \omega_P^{\mathsf{p}(\kappa)}   \implies \phi_P \in C_{\kappa} \Theta_1  (I_P,\xi), \qquad 
\xi \in \omega_P^{\mathsf{p}(\kappa)} \setminus \omega_P \implies \phi_P \in  C_{\kappa} \Theta_0  (I_P,\xi).
\]
 The next lemma is a restatement of \cite[Lemma 5.2]{DPOu}.  
\begin{lemma} \label{decompwvltslemma} Suppose $\imath\in \{0,1\}$,  $  \phi \in \Theta^{3M}_\imath(I,\xi) $, and $K\geq 1$. Then  \[
\phi= \psi+ K^{-M} \upsilon, 
\qquad \psi,\upsilon \in C_M\Theta^{M}_\imath(I,\xi), \qquad  
\supp \psi \subset K I.
\]   
\end{lemma}
\begin{remark} \label{waveletdecomp} Let $P$ be a tile, and suppose either $\imath=1$, $  \xi \in {\kappa} \omega_P   $ or $\imath=0$, $  \xi \in {\kappa} \omega_P \setminus \omega_P  $. Then 
  Lemma \ref{decompwvltslemma} may be iterated to deduce the expansion of $\varphi_P\in \Phi^{3M}(P)$
\begin{equation}
\label{e:refwpsreldec}
  \varphi_P = \sum_{k\geq 0 } 2^{-Mk} \varphi_{P,k,\xi}, \qquad \varphi_{P,k,\xi}\in C_M \Theta_\imath^M (I_P,\xi), \qquad  \supp\varphi_{P,k,\xi}  \subset 2^k I_P.
\end{equation}
The expansion \eqref{e:refwpsreldec} is the form of Lemma \ref{decompwvltslemma} we will use in the sequel.
\end{remark}

\subsection{Space-frequency decomposition on minimal tiles} Our aim in this paragraph is to provide a space-frequency decomposition  induced by a  finite collection of spatial intervals $\mathcal I\subset \mathcal D$, where $\mathcal D$ is a fixed dyadic grid. The definition also involves a dilation factor $K\geq 1$. The spatial components of the forthcoming decomposition 
 will come from the collection
 \begin{equation}  \label{czintervals} 
	\mathrm{CZ}_K(\mathcal{J})\coloneqq \text{maximal elements of }  \left\lbrace   G \in \D  :9K^2G \not \supset J \text{ for all } J \in \mathcal{J} \right\rbrace \end{equation}
When clear from context, the subscript  $ K $  is dropped from the  notation. The following properties, which will be of use to us below, are  straightforwardly deduced from \eqref{czintervals}.	\begin{itemize}
		\item[(i)] If $  O \subset \R  $ is an open set, then $ \mathrm{CZ}_K(\mathcal{J},O)\coloneqq \{ G \in \mathrm{CZ}_K(\mathcal{J})  :  G \subset O \} $ partitions $ O $ up to a set of zero measure.
		\item[(ii)] \label{3goverlap} The collection $ \left\{ 3G:G \in \mathrm{CZ}_K(\mathcal{J}) \right\}$ has finite overlap.
		\item[(iii)] If $ J \in \mathcal J,  G \in \mathrm{CZ}_K(\mathcal{J}) $ and $  G \not \subset 9KJ $ then $  G \subset \R \setminus 3KJ $.
		\item[(iv)]  If $ J \in \mathcal{J}, G \in \mathrm{CZ}_K(\mathcal{J})$    and $ G \subset 3KJ $ then $  K \ell_G \leq \ell_J$.
		\item[(v)]  whenever $ h \in L^\infty_0(\R) $ say, there holds  \begin{equation}
\label{e:mutualcont}
\begin{split}
		\sup_{J \in \mathcal{J}} \inf_J \mathrm{M} h \lesssim  \sup_{G \in \mathrm{CZ}_K(\mathcal{J})} \inf_{G} \mathrm{M} h    \lesssim K^2 \sup_{J \in \mathcal{J}} \inf_J \mathrm{M}h.
	\end{split}\end{equation}
	\end{itemize} 
The corresponding collection of \emph{minimal space-frequency tiles} is then defined by
 \begin{equation}  \label{e:cztiles}  \mathbb{M}=\mathbb{M}(\mathcal{J}) \coloneqq \left\{G \times \left[\xi,\xi+\textstyle\frac{1}{\ell_G}\right) : G \in \mathrm{CZ}(\mathcal{J}), \xi \in \textstyle\frac{\Z}{\ell_G}\right\}\subset \mathbb S_{\mathcal D \times \mathcal D_0}.\end{equation}
It is clear that $ \mathbb{M} $ depends on $ \mathcal{J} $, but we choose to keep the latter implicit in the notation when clear from context.
Pick  $ \eta \in \mathcal{S}(\R) $ with $ \supp \eta \subset (-1,1) $ and $ \eta(0)=1 $. For $P\in \mathbb{M}$ define the approximate projection operator $\Pi_P$, acting on $f\in L^2(\R) $ by
\[
\Pi_P f \coloneqq  [f\cic{1}_{I_P}] * \eta_P, \qquad  \eta_P \coloneqq \Mod_{\mathrm{inf}\omega_P} \Dil_{\scl(P)}^1 \eta.
\]  The Poisson summation formula tells us that 
\begin{equation}
\label{e:sfd}
f=
\sum_{P \in \mathbb{M}} \Pi_P f 
\end{equation}
 with convergence in $L^2(\R)$ and almost everywhere. 
 The decomposition \eqref{e:sfd} is approximately space-frequency localized in the sense that
\begin{equation}
\label{e:sfd2}\supp \Pi_Pf \subset  3I_P, \qquad 
\Pi_P f \in C [f]_{1,{P}} \Theta_1(I_P,c_{\omega_P})
\end{equation}
for some absolute constant $C$.
The approximate projection onto the space-frequency region associated to some  $ \mathbb{W}\subset \mathbb{M} $ is then defined, for say $  f \in L^2(\R) $,   by \[ \Pi_{\mathbb{W}}f \coloneqq \sum_{P\in \mathbb{W}} \Pi_P f.\]
	Below, whenever $\mathbb W\subset \mathbb M$, by  \begin{equation}
	\label{e:WG} \mathbb{W}[G]=\left\{P \in \mathbb{W}: I_P=G  \right\} \end{equation}  we indicate the tiles of $\mathbb W$ having a fixed spatial interval  $ G \in \mathrm{CZ}_K(\mathcal{J}) $.  
  \subsection{Main line of proof of Theorem \ref{mainemb}} This paragraph reduces Theorem \ref{mainemb} to a Calder\'on-Zygmund type decomposition of $f$ with respect to an arbitrary family of top data. Details are as follows. To prove the estimate of Theorem \ref{mainemb}, having   fixed
 \[
 \mathbb P\subset \mathbb S^J\;\textrm{finite,}\qquad 
  \varnothing \subsetneq A\subset \mathbb S^J, \; {\mu^{J,\kappa}}(A) \eqqcolon N<\infty,
  \]
  we need to prove the control
 \begin{equation}
\label{e:maincontrolmet}    \left\| W[f] \cic{1}_{\mathbb{P} \cap A} \right\|_{L^{2,\infty}({J,\kappa,\size}_{2,\star})} \lesssim  N^{\frac{1}{2}-\frac{1}{tp'}} [f]_{p,\mathbb P}
\end{equation} for $f\in L^\infty_0(\R)$. If $  N \leq 1 $ then \eqref{e:maincontrolmet}  follows immediately from an application of Proposition \ref{sizelemmacompressed} and \eqref{e:usefulp}. We deal with the difficult case $ N > 1. $
To do so, we   select an almost optimal  collection of trees  $ \mathcal T \subset  \mathcal{T}^{J,\star} $ 
covering $A$, that is
\[   A\subset \bigcup_{T \in \T} T, \qquad  	   \sum_{T \in \T}|I_T| \leq 2 N|J|,   \] 
and denote by $\mathcal F=\{(I_T,\xi_T): T\in \mathcal T\}$ the corresponding collection of top data. Relying on the collection $\F$, for a given $f\in L^\infty_0(\R)$, we   produce the decomposition  
\begin{align}
\label{e:mainCZ0} &\;f=g+b, \\
\label{e:mainCZ2}
& \left\|g\right\|_2 \leq C_{t}  |J|^{\frac12}N^{\frac{1}{2}-\frac{1}{tp'}}  [f]_{p,\mathbb{P}}  \\ &   \left\| W[b] \cic{1}_{\mathbb{P} \cap A }\right\|_{L^{\infty}({J,\kappa,\size}_{2,\star})}   \leq C_t  N^{-\frac{1}{p'}}[f]_{p,\mathbb{P}} 
\label{e:mainCZbad}
\end{align} 
where the constant $C_{t}$ depends only on the fixed parameter $t>1$ and is allowed to vary at each occurrence.
With \eqref{e:mainCZ2}-\eqref{e:mainCZbad} in hand, we use quasi-subadditivity of the   $L^{2,\infty}({J,\kappa,\size}_{2,\star})$-quasinorm to obtain
\[
\begin{split}
\left\| W[f] \cic{1}_{\mathbb{P} \cap A} \right\|_{L^{2,\infty}({J,\kappa,\size}_{2,\star})} &\leq 
2\left\| W[g] \cic{1}_{\mathbb{P} \cap A} \right\|_{L^{2,\infty}({J,\kappa,\size}_{2,\star})} + 2\left\| W[b] \cic{1}_{\mathbb{P} \cap A} \right\|_{L^{2,\infty}({J,\kappa,\size}_{2,\star})}
\\ & \leq 2 \left\| W[g]   \right\|_{L^{2,\infty}({J,\kappa,\size}_{2,\star})} + 4 
\left\| W[b] \cic{1}_{\mathbb{P} \cap A} \right\|_{L^{\infty}({J,\kappa,\size}_{2,\star})} N^{\frac12}
\\
  &\leq C|J|^{-\frac12}\left\|g\right\|_2   
 + C_t  N^{\frac12-\frac{1}{p'}}[f]_{p,\mathbb P} \leq C_t N^{\frac12-\frac{1}{tp'}} [f]_{p,\mathbb P}.
\end{split}
\]
To pass to the second line, we have employed monotonicity on both terms  and  \eqref{e:usefulp}. The subsequent bound follows from an application of Proposition \ref{usual}, in particular \eqref{usual2} and by taking advantage of \eqref{e:mainCZbad}, while the final estimate is a consequence of \eqref{e:mainCZ2}. This completes the proof of \eqref{e:maincontrolmet}, and in turn of Theorem \ref{mainemb}, up to actual construction of the splitting $f=g+b$ with properties \eqref{e:mainCZ2}-\eqref{e:mainCZbad}. This task is conducted in the upcoming paragraphs.
The first step towards \eqref{e:mainCZ2}-\eqref{e:mainCZbad}  is to construct a suitable collection of  {minimal space-frequency tiles} adapted to the collection $\mathbb P$. To do so, take  \[ \mathcal{J}=\left\{J \in \D:  J=I_P \, \textrm{for some}\, P\in \mathbb P\right\} \] in \eqref{czintervals}. The choice of the constant $K\geq 1$ depends on $N,p$ and $t$ and will be made explicit in \eqref{e:choiceK} below.
From now on, $\mathbb M=\mathbb M(\mathcal J)$ refers to the  collection obtained from \eqref{e:cztiles} for this choice of $\mathcal J,K$. Note that the spatial components of the tiles in $\mathbb M$ come from the collection $   \mathrm{CZ}_K(\mathcal{J}) $. This fact will be employed in the proof quite a few times.

Below, the notation $ T $ is used, with meaning clear from context, for both the top data pair itself $ T= (I_T,\xi_T)\in \F $ and to the set $ T= T(I_T,\xi_T)=  \{P \in \mathbb{P}: I_P \subset I_T, \; \xi_T \in \omega_{P}^{\mathsf{p}(\kappa)} \} .$ The collection of top data $\F$ induces a certain decomposition of the minimal tiles $\mathbb M$, as follows. First, the \emph{principal region} $\mathbb Q$ is defined by 
\begin{equation} \label{e:pr1}
\begin{split}
&  \mathbb{Q}= \bigcup_{T\in \mathcal F} \left\{Q \in \mathbb{M}:  {\scl{(Q)}}|\inf{\omega_Q}-\xi_T| \leq  {K}, I_Q \subset 3KI_T \right\}.
\end{split}\end{equation}
Each  $ T=(I_T,\xi_T ) \in \F$ then partitions the   \emph{tail region} $\mathbb{M}\setminus \mathbb Q$   into the two components  \begin{equation} \label{e:pr12}\begin{split}\mathbb{Q}'(T)&\coloneqq \left\{Q \in \mathbb{M}\setminus \mathbb{Q} :  \scl(Q) |\inf\omega_Q-\xi_T|>{K} \right\},   \\  \mathbb{Q}''(T)&\coloneqq \left\{Q \in \mathbb{M} \setminus \mathbb{Q}: \scl(Q) |\inf\omega_Q-\xi_T|\leq {K}\right\}
	\end{split} \end{equation}   roughly corresponding  to the frequency tails and spatial tails with respect to $T$. The definitions guarantee that $ \mathbb{M}= \mathbb{Q} \sqcup \mathbb{Q}'(T) \sqcup \mathbb{Q}''(T)$ for each $T\in \mathcal F.  $

\subsection{Space-frequency tail estimates} \label{sec:tailestimates}

 The following technical lemma, via a suitable decomposition, shows how the action of the  (adjoint)  frequency tails projection  $ \Pi_{\mathbb{Q}'(T)}$ on wave packets localized to $T$  is exponentially small in the separation parameter $K$.

\begin{lemma} \label{adjprojdecomp}  Let $M$ be a large integer. There exists a  positive constant $C=C(M)$ and a decomposition \[ \Pi_{\mathbb{Q}'(T)}^* =  \Pi_{\mathbb{Q}'(T)}^{*,\mathsf{avg}} + \Pi_{\mathbb{Q}'(T)}^{*,\mathsf{osc}}  \]  with the  following properties.
  \vskip1mm \noindent \emph{(i)} For each pair $ f,g \in L^2(\R) $, there exists $  h\in L^2(\R) $ such that $ |h| \leq  C |f| $ and  
\begin{align}
	 \label{propertyinfty}  \l f, \Pi_{\mathbb{Q}'(T)}^{*,\mathsf{avg}}g \r = K^{-M}\l h,g \r .
\end{align}
\vskip1mm \noindent \emph{(ii)}
If  $I\in \mathcal D$, the pointwise inequality
 \begin{equation}
	\label{propopolywocpt} \sum_{\substack{P \in T\\ I_P \subset I}} \left|\Pi_{\mathbb{Q}'(T)}^{*,\mathsf{osc}}    \phi_P \right| \leq C  K^{-\frac{M}{10}} \chi_{I}^{\frac{M}{10}}
\end{equation}
holds for each $L^\infty$-normalized collection $\{ \phi_P :P \in T, |I_P|^{-1} \phi_P  \in \Theta_{1}^M(I_P,\xi_T) \}$.
\end{lemma}

\begin{proof}
	 With the notation of  \eqref{e:WG}, \[ \Pi_{\mathbb{Q}'(T)}^{*}=\sum_{G \in \mathrm{CZ}_K(\mathcal{J})} \left[ f* {\big(\overline{{Z_G}}(-\cdot)\big)}\right] \cic{1}_G ,\qquad {Z_G}\coloneqq\sum_{Q \in \mathbb{Q}'(T)[G]}\eta_Q. \]
	The claimed decomposition is   \begin{equation}\label{projdeco}\begin{split}
		&	\Pi_{\mathbb{Q}'(T)}^{*,\mathsf{avg}}f \coloneqq \sum_{G \in \mathrm{CZ}_K(\mathcal{J})} \overline{\widehat{{Z_G}}(\xi_T)} \left(\cic{1}_Gf\right), \qquad \Pi_{\mathbb{Q}'(T)}^{*,\mathsf{osc}} \coloneqq \Pi_{\mathbb{Q}'(T)}^{*,{\mathsf{poly}_n}}+\Pi_{\mathbb{Q}'(T)}^{*,{\mathsf{cancel}_n}} \\
		 & \Pi_{\mathbb{Q}'(T)}^{*,{\mathsf{poly}_n}}f \coloneqq \sum_{G \in \mathrm{CZ}_K(\mathcal{J})} \cic{1}_G \int \Mod_{x}\widehat{f} \overline{\left[ {\mathsf{P}_{n}^{\xi_T}}(\widehat{{Z_G}}) -\widehat{{Z_G}}(\xi_T)\cic{1}_\R\right]} 
		 \\ & \Pi_{\mathbb{Q}'(T)}^{*,{\mathsf{cancel}_n}}f \coloneqq \sum_{G \in \mathrm{CZ}_K(\mathcal{J})} \cic{1}_G \int \Mod_x \widehat{f}\overline{ \left[\widehat{{Z_G}}-{\mathsf{P}_{n}^{\xi_T}}(\widehat{{Z_G}})\right]} 
\end{split}
	\end{equation}  where $  n \coloneqq\frac{M}{10} $ and $ {\mathsf{P}_{n}^{\xi_T}}(\widehat{{Z_G}}) $ is the order $n$ Taylor polynomial of $\widehat{{Z_G}}$ centered at $ \xi_T. $
First, observe that \eqref{propertyinfty} follows by taking advantage of the trivial estimate 
\begin{equation}\label{e:derestg}
\left|\partial^{\nu}(\widehat{{Z_G}})(\xi_T)\right| \lesssim_{M,\nu} K^{-M} (\ell_G)^{\nu}, 
\end{equation} 
and subsequently setting  \begin{equation}\label{avgest}
	h \coloneqq \frac{1}{K^{-M}}\sum_{G \in \mathrm{CZ}_K(\mathcal{J})} \widehat{{Z_G}}(\xi_T)(\cic{1}_Gf) \Rightarrow |h| \lesssim_{M} |f|.
\end{equation}  As an intermediate step towards \eqref{propopolywocpt}, we  first prove a preliminary result  under a temporary  spatial compact support  assumption. Namely, for $  \mathsf{i} \in\left\{ \mathsf{poly_n,cancel_n}\right\} ,$  $ \Pi_{\mathbb{Q}'(T)}^{*,\mathsf{i}} $ has the property that  if $  \varphi_P \in \Theta_{\iota}^M(I_P,\xi_T) $ with $  \supp \varphi_P \subset \Delta I_P $ \begin{align}\label{propertypoly}
	&	 \left|\Pi_{\mathbb{Q}'(T)}^{*,{\mathsf{poly}_n}}\varphi_P \right| \lesssim \ell_{I_P}^{-1} \Delta K^{-M}  
		  \sum_{m=1}^{n} \sum_{G \in \mathrm{CZ}_K(\mathcal{J}),3G \cap \Delta I_P \neq \varnothing } \cic{1}_G  \left(\frac{\ell_G}{\ell_{I_P}}\right)^{m} \\
	& \label{propertycancel}  \left|\Pi_{\mathbb{Q}'(T)}^{*,{\mathsf{cancel}_n}}\varphi_P\right|  \lesssim \ell_{I_P}^{-1} \Delta \sum_{G \in \mathrm{CZ}_K(\mathcal{J}),3G \cap \Delta I_P \neq \varnothing} \cic{1}_G \left(\frac{\ell_G}{\ell_{I_P}}\right)^{n+1}  \\
	\end{align}
We check \eqref{propertypoly}.  The adaptation 
$\left|\widehat{\varphi_P}\right| \lesssim \Delta \chi_{\omega_{\xi_T,P}}^{M}(\xi)$ allows us to estimate  \begin{equation}\label{calculation}
	 \left\|\left(\xi-\xi_T\right)^{j} \widehat{\varphi_{P}}\right\|_{1}\lesssim  \left( \int_{|\xi-\xi_T| \leq  \ell_{\omega_P}}|\xi-\xi_T|^j+\int_{|\xi-\xi_T|> \ell_{\omega_P}} |\omega_P|^{M} |\xi-\xi_T|^{j-M}\right) \lesssim K  (\ell_{I_P})^{-(j+1)}     \end{equation} for $ 1 \leq j \leq n+1 $.
 Combining \eqref{calculation} with \eqref{e:derestg} yields \eqref{propertypoly}. Finally, an application of \eqref{calculation} for $ j=n+1 $ yields \eqref{propertycancel}.
In order to prove \eqref{propopolywocpt}, apply   Remark \ref{e:refwpsreldec} to write $ \phi_P $ as rapidly decaying superposition of wave packets with compact support and use the intermediate estimate \eqref{propertypoly}  as in  \begin{equation}\label{oscterm}
	 \begin{split}
&	\quad \sum_{k \geq 0} 2^{-\frac{M}{3}k}\sum_{P \in T} \left|\Pi_{\mathbb{Q}'(T)}^{*,\mathsf{osc}}\ \phi_{P,k,\xi_T}  \right|\\ & \lesssim  \sum_{k \geq 0}2^{-\frac{M}{3}k} 2^{k(n+1)}  \sum_{P \in T} \sum_{\substack{G \in \mathrm{CZ}_K(\mathcal{J}) \\ 3G \subset 3 \cdot2^kI_P}}\cic{1}_G \max\left\{K^{-M}, \left(\frac{\ell_G}{2^k \ell_{I_P}}\right)^{n} \right\} \left(\frac{\ell_G}{2^k \ell_{I_P}}\right) \\ & \lesssim K^{-\frac{M}{10}} \sum_{k \geq 0}2^{-\frac{M}{5}k} \sum_{\substack{G \in \mathrm{CZ}_K(\mathcal{J}) \\ 3G \subset 6 \cdot 2^k I }} \cic{1}_G \sum_{\substack{P \in T \\ 3G \subset 3 \cdot 2^k I_P}} \left(\frac{\ell_G}{2^k \ell_{I_P}}\right).
\end{split}
\end{equation}
The proof of \eqref{propopolywocpt} is then completed by summing up, and taking advantage of the next two observations.  First, when $G \in \mathrm{CZ}_K(\mathcal{J})$ and $j \in \Z $ the   cardinality estimate \[\# \left\{P \in T: 3G \subset 3 \cdot 2^k I_P, \scl(P)=2^j \right\} \lesssim (k+2)2^k \] holds  uniformly in   $ j,G $. Next, when $ G \in \mathrm{CZ}_K(\mathcal{J}) $, $ J \in \mathcal{J} $ and  $  3G \cap 2^kJ \neq \varnothing$,  then  $  G \subset 3 \cdot  2^k J $ necessarily,  and in particular  $  3K^2 {\ell_{G}} \leq {\ell_{2^kJ}}   .$ Therefore,  the counting  estimate in the last display allows us to perform a single scale analysis in the innermost sum of \eqref{oscterm}, and using the disjointness of $  G \in \mathrm{CZ}(\mathcal{J}) $, we can estimate \eqref{oscterm} by \begin{equation}\label{oscfinal}
	 K^{-\frac{M}{10}} \sum_{k \geq 0} 2^{-\frac{M}{6}k} \cic{1}_{2^{k+3}I} \lesssim K^{-\frac{M}{10}} \chi_{I}^{\frac{M}{10}}
\end{equation}  Finally, the proof of the lemma  is finished  by taking $  C(M) $ to be the larger of the two implied constants   in \eqref{avgest}, \eqref{oscfinal}.  
\end{proof}
\begin{lemma} \label{off}  If $  f \in L^{\infty}_0(\R) $ and  $ T\in \mathcal F $ there holds \[ \size_{2,\star,\kappa}(W[\Pi_{\mathbb{M} \setminus\mathbb{Q}}(f)],T   ) \lesssim K^{-\frac{M}{20}} [f]_{1,\mathbb{P}}. \]
	
\end{lemma}
\begin{proof} The proof is carried by splitting $ \Pi_{\mathbb{M}\setminus \mathbb{Q}} $   into frequency and spatial tails. Namely, \[ \Pi_{\mathbb{M}\setminus \mathbb{Q}}f=\Pi_{\mathbb{Q}'(T)}f+\Pi_{\mathbb{Q}''(T)}f,\] and  it suffices to check that \[ \max_{\mathsf{i} \in \left\{\mathsf{',''}\right\}} \size_{2,\star,\kappa}(W[\Pi_{\mathbb{Q}^{\mathsf{i}}(T)}f],T  ) \lesssim K^{-\frac{M}{20}} [f]_{1,\mathbb{P}}  \]
	The case $ \mathsf{i}=\mathsf{'} $ is dealt with first. By identical considerations to those from  the proof of Proposition \ref{sizelemmacompressed}, it suffices to bound \begin{equation}\label{JSnew}
		\left\| \sum_{\substack{P \in T' \\ I_P \subset I}}|I_P| \l f , \Pi^{*}_{\mathbb{Q}'(T)}\phi_P \r h_{I_P} \right\|_{1,\infty} \lesssim K^{-\frac{M}{20}}  [f]_{1,\mathbb{P}} |I| 
	\end{equation} for an arbitrary interval $I\in \mathcal D$, lacunary tree $  T' \subset T $,  and collection $\{ \phi_P\in \Phi(P):P\in T'\} $. To this purpose, Lemma \ref{adjprojdecomp} entails\[ \begin{split}
	&	 \left\| \sum_{\substack{P \in T' \\ I_P \subset I} }|I_P| \l f, \Pi_{Q'(T)}^{*,\mathsf{osc}}\phi_P\r h_{I_P} \right\|_{1} \leq \left\langle |f|, \sum_{\substack{P \in T' \\ I_P \subset I} }   \left|\Pi_{Q'(T)}^{*,\mathsf{osc}} [|I_P| \phi_P ] \right| \right \rangle 
	 \lesssim K^{-\frac{M}{20}} \left\| f \chi_{I}^{\frac{M}{10}} \right\|_{1}  \lesssim  K^{-\frac{M}{20}}|I| [f]_{1,\mathbb{P}}.
	\end{split} \]
Furthermore,  \eqref{propertyinfty} of Lemma \ref{adjprojdecomp} 
may be  used to find $|h| \leq C |f|$ such that the estimate  \[	\left\| \sum_{\substack{P \in T' \\ I_P \subset I}}|I_P| \l f , \Pi^{*,\mathsf{avg}}_{\mathbb{Q}'(T)}\phi_P \r h_{I_P} \right\|_{1,\infty}\lesssim K^{-M}|I| [h]_{1,\mathbb{P}} \lesssim K^{-M} |I| [f]_{1,\mathbb{P}} \] holds. This completes the handling of the term $  \mathsf{i}= \mathsf{'} $.
The term $  \mathsf{i}= \mathsf{''} $ is much easier, The definition of $\mathbb Q''(T)$ guarantees that $ \supp \Pi_{\mathbb{Q}''(T)}f \cap KI_T =\varnothing.  $
Therefore, an application of Lemma \ref{sizelemmacompressed} in the first step yields
\[
\begin{split}
\size_{2,\star,\kappa}(W[\Pi_{\mathbb{Q}''(T)}f],T )  \lesssim K^{-\frac{M}{2}} \left \| \Pi_{\mathbb{Q}''(T)}f \right \|_{\infty}.
\end{split}
\]
while properties (ii) and (v) of the  spatial intervals $\mathrm{CZ}_K(\mathcal{J}) $ guarantee the bound
 \[
		\left \| \Pi_{\mathbb{Q}''(T)}f \right \|_{\infty} \lesssim K \sup_{G \in \mathrm{CZ}_K(\mathcal{J})} \inf_{G} M(f) \lesssim K^3 [f]_{1,\mathbb{P}}. 
\] 
The claim of the lemma for the $\mathbb{Q}''(T)$ component   is then an immediate consequence of the last two displays. 
\end{proof}
\subsection{Conclusion of the proof} \label{sec:mfczd}  The choices
\begin{equation}
K\coloneqq N^{\frac{1}{p'}\left(\frac{10}{M}+\frac{1}{5t'}\right)}, \qquad M \coloneqq M(t) = 10 \lceil 2^8t' \rceil  \label {e:choiceK}
\end{equation} and the 
decomposition \eqref{e:mainCZ0} are  now made explicit. The choice of $ M $, anticipated in Remark \ref{smoothnesstprime}, ensures \eqref{e:mainCZ2}, \eqref{e:mainCZbad} both hold. In view of the minimal tiles expansion \eqref{e:sfd},  set in \eqref{e:mainCZ0} \begin{align}
 g \coloneqq \Pi_{\mathbb{Q}}f, \qquad b \coloneqq \Pi_{\mathbb{M} \setminus \mathbb{Q}}f.
\end{align}
Turn to the verification of \eqref{e:mainCZ2}-\eqref{e:mainCZbad}. For the first, write   \[ \Pi_{\mathbb{Q}}f=\sum_{G \in \mathrm{CZ}_K(\mathcal{J})} \Pi_{\mathbb{Q}\left[G\right]}f=\sum_{G \in \mathrm{CZ}_K(\mathcal{J})}\left(f \cic{1}_G\right)*U_G,\quad U_G\coloneqq\sum_{P \in \mathbb{Q}\left[G\right]} \eta_P.\] A straightforward use of Plancherel's theorem entails that $  \left\| \Pi_{\mathbb{Q}\left[G\right]} \right\|_{2 \to 2}  \lesssim 1. $ Furthermore,    observe that                                          $ \widehat{U_G} $ is a sum of $ \#\mathbb{Q}\left[G\right] $ Schwartz functions uniformly  adapted to disjoint intervals of length $ \ell_G^{-1} $, leading to the estimates \[ \left \| \widehat{U_G} \right \|_1 \lesssim \frac{\# \mathbb{Q}[G]}{\ell_{G}}, \quad \left \| \widehat{U_G} \right \|_{\infty} \lesssim 1 
\implies \left \| \Pi_{\mathbb{Q}\left[G\right]} \right \|_{p \to 2} \lesssim \left(\frac{\# \mathbb{Q}\left[G\right]}{\ell_G}\right)^{\frac{1}{p}-\frac{1}{2}}
 \] where the  implication is obtained   by  log-convexity, Young's inequality,  and finally Riesz-Thorin interpolation of the $(2,1)  $ and $ (2,2) $  estimates.   Preliminarily, also  note        \begin{equation}\label{tilecountergood}
	\# \mathbb{Q}\left[G\right] \leq   3K \inf_{G} \sum_{T \in \T} \cic{1}_{3KI_T} \qquad \forall  G \in \mathrm{CZ}_K(\mathcal{J}).
\end{equation} Estimate \eqref{e:mainCZ2} then follows by combining \eqref{tilecountergood}, the fact that $ \supp \Pi_{\mathbb{Q}\left[G\right]}f \subset 3G $ and the finite overlap   {\color{red}(ii)}  of $     \{3G: G\in  \mathrm{CZ}_K(\mathcal{J}) \}$ in the string of inequalities \[ \begin{split}
	\left \| \Pi_{\mathbb{Q}}f \right \|_2 & \lesssim \left(\sum_{G \in \mathrm{CZ}_K(\mathcal{J}),G \subset 9KJ} \left \| \Pi_{\mathbb{Q}\left[G\right]}f \right \|_2^2  \right)^{\frac{1}{2}} \lesssim  |J|^{\frac{1}{p'}}K^{\frac{5}{2}} \left(\sum_{T \in \T}|I_T|\right)^{\frac{1}{2}-\frac{1}{p'}} [f]_{p,\mathbb{P}} \\ & \lesssim |J|^{\frac{1}{2}}N^{\frac{1}{2}-\frac{1}{tp'}}[f]_{p,\mathbb{P}}
\end{split}\] as claimed.
For the property \eqref{e:mainCZbad}, by Lemma \ref{treestructure} it suffices to check that for each $ T \in \T $ there holds \[ \size_{2,\star}(\Pi_{\mathbb{M} \setminus \mathbb{Q}}f,T) \lesssim N^{-\frac{1}{p'}}[f]_{p,\mathbb{P}}. \] Taking notice of the relation between $M $ and $t$ in \eqref{e:choiceK}, this was proved in Lemma \ref{off}.

	\section{Proof of Theorem \ref{t:A}}  \label{s:pcarl} Fix a tiling $\mathbb S= \mathbb S_{\mathcal D,\mathcal D'} $, $f_j\in L^\infty_0(\R)$, $j=1,2$. The crux of the matter is to establish the estimate
\begin{equation}
\label{e:maintApf1}
\sup_{\mathbb P \subset \mathbb S \textrm{ finite}} \mathsf{C}_{\mathbb{P}} (f_1,f_2) \lesssim \frac{1}{\eps} \left\|\M_{\left(\frac{1}{1-\eps},1\right)}(f_1,f_2) \right\|_1
\end{equation}	
with implied constant independent of $\eps>0$,  referring to the model sums \eqref{e:modelsum}. In fact, if $\mathcal C$  stands for  \eqref{carleson}, in view of \eqref{e:modelred}, the form $\langle \mathcal C f_1, f_2\rangle $ is controlled by the sum of   $\lesssim 1$ terms of the type appearing in the left hand side of \eqref{e:maintApf1}.  The same sparse bound for  the periodic operator \eqref{e:mainspper} then follows from \eqref{e:mainsp1} via a standard transference type argument based on the Stein-Weiss lemma, see e.g \cite[Appendix A]{OSTTW}.

We turn to the proof of \eqref{e:maintApf1}, fixing $0<\eps\leq \frac12$, and a finite collection $\mathbb P \subset \mathbb S$. To unify   notation below, it is convenient to write $q_1=\frac{1}{1-\eps}, q_2=1 $ and $\vec q=(q_1,q_2)$ below.
 Let $ \mathcal Q\subset \mathcal D$  be a partition of $ \R $ with the property that
 $ \supp f_j \subset 3Q $ for $ j=1,2$,  $Q \in \mathcal{Q}. $
For each $  Q \in \mathcal{Q}  $, define   $  \S_0(Q)=\left\{Q\right\} $ and  inductively for $m\geq 0$
\[ 	\begin{split}
&\mathcal{B}(S)\coloneqq \text{maximal } B \in \D \textrm{ with }
	B \subset S\cap   \bigcup_{j=1}^2\left\{  \mathrm{M}_{q_j}\left[f_j \cic{1}_{3S}\right ] >\Theta \l f_j \r_{q,3S}\right\}, \quad  S \in \S_{m}(Q), \\ & 
\mathcal S_{m+1}(Q)\coloneqq \bigcup_{S \in \S_{m}(Q)}\mathcal{B}(S). \end{split}  
.\]
Finish by setting $\mathcal S\coloneqq\bigcup_{Q\in \mathcal Q} \bigcup_{m\geq 0} \mathcal S_{m}(Q) $.
Note that the sets $\{E_S:S\in \mathcal S\}$ defined by 
$E_S\coloneqq S\setminus\left[ \bigcup\{B:B \in \mathcal B(S)\}\right]$ are pairwise disjoint, and 
 the packing condition
\[
	\sum_{B\in \mathcal{B}(S)}|B| \leq \frac{1}{4}|S|, \qquad S\in \mathcal S 
\]
which holds provided   the absolute constant $ \Theta $ is picked suitably large, guarantees that $|S|\leq 2|E_S|$ for all $S\in \mathcal S$. Also, the iterated stopping interval nature of the collection $\mathcal S$ yields  
   \begin{equation}\label{packingstopping}
	\sup_{\substack{I \subset S \\ I \not \subset \bigcup_{B\in \mathcal B(S)}  B } }  \inf_{I}\mathrm{M}_{q_j}f_j  \lesssim  \l f_j \r_{q_j,3S}, \qquad S\in \mathcal S, \; j=1,2.
\end{equation}
Therefore, the partition
\[
\mathbb P=\bigsqcup_{S\in \mathcal S} \mathbb P(S), \qquad \mathbb P(S)\coloneqq \left\{P\in \mathbb P : I_P \subset S, I_P \not \subset  \bigcup_{B\in \mathcal B(S)} B  \right\}
\]
inherits the property
 \begin{equation}\label{e:cons}
  [f_j]_{q_j,\mathbb P(S) } \lesssim \l f_j \r_{q_j,3S}, \qquad S\in \mathcal S, \; j=1,2.
\end{equation}
By virtue of \eqref{e:carlholder}, we may apply  Proposition \ref{p:eHolder} with $  p_1=2q_1'=\frac{2}{\eps}$, $ a=2$ and $ p_2=1 $ in the second step, and pass to the second line through an appeal to Theorem \ref{mainemb} with $t=2$ and $p=q_1$ for $W$, and  Proposition \ref{e:modifemb1} for $A$, obtaining 
\[
\begin{split}
\mathsf{C}_{\mathbb{P}}(f_1,f_2) &
= \sum_{S\in \mathcal S} \mathsf{C}_{\mathbb{P}(S)}(f_1,f_2)  
 \lesssim \frac{1}{\eps }
\sum_{S\in \mathcal S}  |S|
\left \| W[f_1] \cic{1}_{ \mathbb{P}(S)} \right \|_{X_2^{2q_1',\infty}(S,\size_{2,1,\star})}
 \left \| A[f_2] \cic{1}_{ \mathbb{P}(S)} \right \|_{Y^{1,\infty}(S,\size_{\mathsf{C}})}
 \\ & \lesssim  \frac{1}{\eps} \sum_{S\in \mathcal S}   |S| \prod_{j=1,2}[f_j]_{q_j,\mathbb P(S) }\lesssim \frac{1}{\eps} \sum_{S\in \mathcal S}   |E_S| \inf_{E_S} \M_{\vec q}(f_1,f_2) \leq { \frac{1}{\eps}}  \|\M_{\vec q}(f_1,f_2) \|_1.
\end{split}
\]
The middle almost-inequality in the last line  relies on \eqref{e:cons} as well as $|E_S|\gtrsim {  | S  | }$, while the final step is due to the pairwise disjointness of $E_S$, $S\in \mathcal S$. The proof is thus complete.

	\section{Proof of Theorem \ref{multipliersparse}} \label{s:pfbht}
\subsection{Rank 1 forms} This paragraph devises a reformulation, within our framework, of the trilinear forms discretizing multipliers with singularity along a rank one subspace, such as the bilinear Hilbert transforms. Although these date back in essence to the  works of Lacey-Thiele \cite{LTbht1,LTbht2},  they appear in a form closer to ours in \cite{MTT1}. The main change in our definition with respect to the usual one is that our does not involve multi-tiles, at least explicitly, to avoid reformulating outer $L^p$-spaces and use our embedding theorems in the most direct way possible.

Fix $\kappa \geq 1$,  two dual dyadic grids $\mathcal D,\mathcal D'$ and the tiling $\mathbb S=\mathbb S_{\mathcal D,\mathcal D'}$. Our construction, similarly to \cite{MTT1},  relies on the two order relations on $\mathbb S$
\begin{align}
P\lesssim_\kappa  P' &\iff I_P \subset I_{P'}, \;  \omega_{P'}^{\mathsf{p}(\kappa)} \subset   \omega_{P}^{\mathsf{p}(\kappa)}, \qquad \kappa\geq 1, \\
P\lesssim'_\kappa  P'& \iff  P\lesssim_\kappa P' \textrm{ and } P\not \lesssim_{1} P', \qquad \kappa \geq 2.
\end{align}
Note that $\lesssim_\kappa$  has been already defined in  \eqref{e:feffor} and is recalled here for the reader's convenience.   
Let $\kappa\geq 10$,  $\mathbb P$ be a finite subset of $\mathbb S$ with scales separated by a factor of $2^{5\kappa}$, and $\eta=(\eta_1,\eta_2,\eta_3): \mathbb P \to  \mathbb S \times \mathbb S \times  \mathbb S $ have the properties
\begin{itemize}
\item[r1.]   the components $\eta_j :\mathbb P \to  \mathbb S $ are injective maps for $j=1,2,3$;
\item[r2.]   $I_{ \eta_j(P)}=I_P$ for $j=1,2,3$;
\item[r3.] if $P,P'\in \mathbb P$ are such that $\eta_j(P)\lesssim_1 \eta_j(P') $ for some $j\in\{1,2,3\}$ then $\eta_k(P)\lesssim_\kappa \eta_j(P') $ for all $k\in\{1,2,3\}$;
\item[r4.] if $P,P'\in \mathbb P$ are such that $\eta_j(P)\lesssim_1 \eta_j(P') $ for some $j\in\{1,2,3\}$ then $\eta_k(P)\lesssim_\kappa \eta_k (P') $ for all $k\in\{1,2,3\}$, and in fact $\eta_k(P)\lesssim'_\kappa \eta_k (P') $ for at least two indices $k\in\{1,2,3\}$.
\end{itemize}
It is convenient to denote by $\mathbb P_j$, $j\in\{1,2,3\}$ the ranges of $\eta_j$. The  \emph{rank 1 form} of parameter $\kappa$ associated to $\eta$ and $\mathbb Q\subset P$ acts on a triple $f_j\in L^\infty_0(\R)$ by
\begin{equation}
\label{e:tileform}
\Lambda_{\eta,\mathbb Q} (f_1,f_2,f_3)=\sum_{P\in \mathbb Q} |I_P|\prod_{j\in\{1,2,3\}} W[f_j](\eta_j(P))
\end{equation}
where $W$ stands for the wave packet transform \eqref{e:wpt}. 

A typical example of map $\eta$ satisfying r1.\ to r4. and thus  giving rise to   rank $1$ forms is the following.
Let $\Gamma'$ be a  1-dimensional  subspace of $\Gamma=\{\xi\in\R^3: \xi_1+\xi_2+\xi_3 =0\}$ as in the statement of Theorem \ref{multipliersparse} and $\mathcal Q\subset \mathcal D'\times \mathcal D' \times \mathcal D'$ be a finite collection satisfying 
\begin{itemize}
\item[g1.] $\ell_{Q_1}= \ell_{Q_2}=\ell_{Q_3}\in   2^{H\mathbb Z+ h }$ for all $Q=Q_1\times Q_2 \times Q_3\in \mathcal Q$,
\item[g2.] $Q\cap \Gamma \neq \varnothing $ for all $Q\in \mathcal Q$,
\item[g3.] $K \ell_{Q_1}\leq \dist(Q, \Gamma')\leq K^2 \ell_{Q_1}$ for all $Q\in \mathcal Q$.
\end{itemize}
for parameters   $H,K\in \mathbb N$ and $h\in\{0,\ldots, H-1\}$. If $H,K$ are sufficiently large parameters depending on $\Gamma'$,   conditions g1.\ to g3.\ tell us that the collection $\{Q\in \mathcal Q: Q_1=\omega\}$ has at most one element for each $\omega\in \mathcal D'$, see e.g.\ \cite[Lemma 6.2]{MTT1}. If such collection is nonempty, we may then write $Q=\omega \times Q_2(\omega)\times Q_3(\omega)$ for its unique element. Of course, the index   $j=1$ can be replaced by any other index in a symmetric statement.
In this setting, if $\mathbb P=\mathbb P_1$ is a finite subset of $\{P\in \mathbb S: \omega_{P}=Q_1$ for some $Q\in \mathcal Q\}$ the map
\[
\eta:\mathbb P\to \mathbb S \times \mathbb S \times \mathbb S, \qquad \eta(P)=\left( P, I_P\times Q_2(\omega_P),  I_P\times Q_3(\omega_P)\right)
\]
satisfies  r1.\ to r4. The usual model sum reduction of \cite{MTT1} may be then summarized in the statement that the singular multipliers \eqref{e:BHTmod} lie in the closed convex hull of rank 1 forms as defined above, with parameter $\kappa $  chosen sufficiently large   depending on the parameter $K$ in g3. Therefore Theorem \ref{multipliersparse} will follow from the estimate
\begin{equation} \label{e:tilezzz}
 \Lambda_{\eta,\mathbb P}(f_1,f_2,f_3) \leq   \frac{C}{\eps(\vec p)}    \left \| \mathrm{M}_{ \vec{p} }(f_1,f_2,f_3) \right \|_1
\end{equation}
uniformly over all rank 1 forms, for all tuples $\vec p=(p_1,p_2,p_3)$ satisfying the conditions in \eqref{e:BHTrange}. 
Symmetry in the indices $p_1,p_2,p_3$ and a   complex interpolation argument  allow us to restrict ourselves to tackling \eqref{e:tilezzz} in the extremal case $p_1=\frac{1}{1-\eps}, p_2=2=p_3$, for which $\eps(\vec p)=\eps.$  We do so in the next paragraph.
\subsection{Using the wave packet embeddding} We now prove \eqref{e:tilezzz} in the above mentioned extremal case.
By eventually composing $\eta$ with the inverse of $\eta_1$, we   reduce to the case where $\mathbb P=\mathbb P_1$ and $\eta_1$ is the identity map.
Properties r1.\ to r4.\ of the map $\eta$ associated to a rank 1 form  of fixed constant $\kappa$ come into play via the following observation.
If $\varnothing \subsetneq S\subset \mathbb  S$, $Q\in S$, then the set \[
S(Q)\coloneqq\{P\in S: P\lesssim_1 Q\}\] is a $1$-tree  with top $(I_{Q},c_{_{\omega_Q}})$. Property r3.\ tells us that the sets 
$
\eta_j(S(Q))\coloneqq\{ \eta_j(P):P\in \mathbb S(Q) \}
$
are $\kappa$-trees with top $(I_{\eta_j(Q)},c_{_{\omega_{\eta_j(Q)}}}) = (I_{Q},c_{_{\omega_{\eta_j(Q)}}})  $, $j=1,2,3$. Furthermore, property r4.\ and scale separation as in the proof of Lemma \ref{lacpartislac} allows us to decompose
\[
S(Q)=\bigcup_{j=1}^3 S(Q,j)
\]
with each $S(Q,j)$ having the property that  $\eta_k(S(Q,j))$, obviously contained in $\eta_k(S(Q))$, is a lacunary $\kappa$-tree with top $(I_{Q},c_{_{\omega_{\eta_k(Q)}}}) $ for $k\in \{1,2,3\}\setminus \{j\}$. 
In accordance with this property, we define  three new variants of \eqref{e:sizestar}
 on the outer measure space $(\mathbb S^J,\mathcal T^{J,1},{\mu^{J,1}}) $. Setting  for $k=1,2,3$
\[
\size_{2,\star,k}(F, T)\coloneqq \sup\left\{\mathsf{size}_{2}(F\circ \eta_k^{-1}, U): U\subset \eta_k(T), U \textrm{ lacunary } \kappa\textrm{-tree} \right\}, \qquad T \in {{\mathcal T^{J,1}}}
\]
we have the estimate
\begin{equation}
\label{e:holderbht}\size_{1}(F_1F_2F_3, T) \lesssim \prod_{k=1,2,3}\size_{2,\star,k}(F_k, T), \qquad \forall T \in {{\mathcal T^{J,1}}}.
\end{equation} 
This  inequality, proved at the end of this section, is our analogue of the usual \emph{tree estimate}, see e.g.\ \cite[Lemma 7.3]{MTT1}, and it is essentially the only additional piece of machinery we were left to set up. Indeed, if $\mathbb Q\subset P\cap \mathbb S^J$ is arbitrary, \eqref{e:holderbht} allows us to appeal to Proposition \ref{p:eHolder} with the obvious choice of exponents, and obtain the chain of inequalities
\begin{equation}
\label{e:wtih}
\begin{split} &\quad 
\Lambda_{\eta,\mathbb Q} (f_1,f_2,f_3) \leq \left\|  \cic{1}_{\mathbb Q}  \prod_{k=1}^3  \left(W[f_k]\circ\eta_k\right)\right\|_{\ell^1(\mathbb S^J)} \\ & \lesssim \frac{|J|}{\eps} \left\|  \cic{1}_{\mathbb Q}    \left(W[f_1]\circ\eta_k\right)\right\|_{X_2^{\frac{2}{\eps},\infty}(\mathbb S^J,1,\size_{2,\star,1})}  \prod_{k=2,3}  \left\|  \cic{1}_{\mathbb Q}    \left(W[f_k]\circ\eta_k\right)\right\|_{Y^{2,\infty}(\mathbb S^J,1,\size_{2,\star,k})}
\\ & \leq  \frac{|J|}{\eps} \left\|  \cic{1}_{\eta_1(\mathbb Q)}    \left(W[f_1]\right)\right\|_{X_2^{\frac{2}{\eps},\infty}(\mathbb S^J,\kappa,\size_{2,\star})}  \prod_{k=2,3}  \left\|  \cic{1}_{\eta_k(\mathbb Q)}    \left(W[f_k]\right)\right\|_{Y^{2,\infty}(\mathbb S^J,\kappa,\size_{2,\star})}
\\ & \lesssim  \frac{|J|}{\eps}[f_1]_{\frac{1}{1-\eps},\mathbb Q}\prod_{k=2,3}  [f_k]_{2,\mathbb Q}.
\end{split}
\end{equation}
The passage to the third line follows by transport of structure, while for the subsequent step we have applied Theorem \ref{mainemb} and  estimate \eqref{e:2emb}, and used that the spatial components, and thus the corresponding local tile norms on $\mathbb Q$, are invariant under $\eta$.
With \eqref{e:wtih} in hand, a stopping procedure akin to that devised in Section \ref{s:pcarl} easily leads to \eqref{e:tilezzz}. Details are left to the interested reader.
\begin{proof}[Proof of \eqref{e:holderbht}] Let  $T \in {{\mathcal T^{J,1}}}$ be a 1-tree with top $(I_T,\xi_T)$, and $\mathsf{m}(T)$ be the set of those $Q\in T$ which are maximal with respect to $\lesssim_1$. As $\xi_T\in \omega_{P}^{\mathsf{p}(1)}$ for all $P\in T$, it must hold that $I_{Q}\cap I_{Q'}=\varnothing$ whenever  $Q,Q'\in \mathsf{m}(T)$ with $Q\neq Q'$. Clearly, $T$ is the disjoint union of the 1-trees $\{T(Q):Q\in \mathsf{m}(T)\}$. Simply from the definitions and the disjointness we just stressed 
\[
\begin{split} &\quad
\size_{1}(F, T)  \leq \sum_{Q\in  \mathsf{m}(T) }\frac{  |I_Q|\size_{1}(F_1F_2F_3, T(Q)) }{|I_T|} \leq \sup_{Q\in  \mathsf{m}(T) }\size_{1}(F, T(Q)), 
\\  &\sup_{Q\in  \mathsf{m}(T) }\size_{2,\star,k}(F, T(Q)) \leq \size_{2,\star,k}(F, T) \qquad \ 
\end{split}
\]
where the second bound holds for $ k\in\{1,2,3\} $  and follows by obvious inclusion considerations. The last two inequalities tell us that it suffices to prove \eqref{e:holderbht} for $T=T(Q)$. In that case,
\[
\begin{split}
&\quad \size_{1}(F_1F_2F_3, T(Q))  \leq \sum_{j=1}^3  \size_{1}(F_1F_2F_3, T(Q,j)) \\ &\leq  \sum_{j=1}^3  \size_{\infty}(F_j, T(Q,j)) \prod_{k\neq j} \size_{2}(F_k, T(Q,j)) \\ &= \sum_{j=1}^3 \sup_{P\in  T(Q,j) } \size_{2}\left(F_j\circ\eta_{j}^{-1},\eta_j( \{P\}) \right)\prod_{k\neq j} \size_{2}\left(F_k\circ\eta_{k}^{-1},\eta_k( T(Q,j))\right)
\\ &\leq 3  \prod_{k=1^3} \size_{2,\star,k}\left(F_k, T(Q)\right)
\end{split}
\]
as desired. We have used in the last step the lacunarity of $\eta_k( T(Q,j))$  and of the single tile trees  $\{\{P\}:P\in T(Q,j)\}$.  
\end{proof}

 \bibliography{WavePacketAnalysis}{}
\bibliographystyle{amsplain}
\end{document}